\documentclass[smallextended]{svjour3}  

\usepackage{graphicx}
\usepackage{amssymb,amsmath}
\usepackage[hidelinks]{hyperref}
\usepackage{mathtools}
\usepackage{bm}
\usepackage{amscd}
\usepackage{stmaryrd}
\usepackage{mathrsfs}
\usepackage{tikz}
\usepackage{booktabs}
\usepackage{hyperref}
\usetikzlibrary{calc,arrows.meta}
\usetikzlibrary{angles,quotes}
\usetikzlibrary{intersections}

\numberwithin{equation}{section}
\numberwithin{table}{section}
\numberwithin{figure}{section}
\numberwithin{theorem}{section}
\numberwithin{lemma}{section}

\DeclarePairedDelimiter{\jump}{\llbracket}{\rrbracket}
\smartqed

\begin{document}
\title{An Adaptive Finite Element Method Based on Generalized Barycentric Coordinates}


\author{Yihui Zhou, Yuwen Li*}


\institute{Y. Zhou \at
              School of Mathematical Sciences, Zhejiang University, 866 Yuhangtang Road, Hangzhou, Zhejiang 310058, People's Republic of China. \\
              \email{22335023@zju.edu.cn}           
           \and
           Y. Li \at
              School of Mathematical Sciences, Zhejiang University, 866 Yuhangtang Road, Hangzhou, Zhejiang 310058, People's Republic of China.\\
              \email{liyuwen@zju.edu.cn}
}

\date{Received: 2026/04/22}

\maketitle

\begin{abstract}
This work derives a posteriori error estimate of polygonal finite element methods based on Wachspress barycentric coordinates. In particular, we prove that the classical residual-based a posteriori error estimator is both an upper and lower bounds for the discretization error. The analysis relies a Scott-Zhang type interpolation and homogeneity arguments for rational functions on polygonal elements. Numerical experiments on square and L-shaped domains demonstrate the effectiveness of the adaptive algorithm.

\keywords{polygonal mesh\and Wachspress barycentric coordinate\and residual-based error estimator\and adaptive finite element method}
\end{abstract}

\maketitle

\markboth{Y. Zhou and Y. Li}{AFEM Based on GBC}

\section{Introduction}
Polygonal meshes have become an important tool in the numerical approximation of partial differential equations, due to their flexibility in handling complicated geometries, local mesh reconstruction, and evolving interfaces. Compared with triangular and quadrilateral meshes, polygonal meshes can approximate irregular boundaries more naturally and are particularly effective in applications involving crack propagation and branching. This flexibility has led to several numerical frameworks on polygonal meshes, including virtual element method \cite{BeiraoBrezziMariniRusso2023},  discontinuous Galerkin methods \cite{ArnoldBrezziCockburnMarini2001,CockburnGopalakrishnanLazarov2009,WangYe2014},  hybrid high order method \cite{DiPietroErn2015}, discrete de Rham method \cite{DiPietro2023}, and finite elements (FEs) based on generalized barycentric coordinates (GBC) \cite{Wachspress1975,Joshi2007,Hormann2008,Floater2010,Weber2012,Floater2016}. In this paper, we shall focus on a posteriori error analysis of GBC-based FEs on polygonal meshes.  

A posteriori error estimates are central to adaptive mesh refinement for resolving singularities of PDEs. Existing approaches in this field include residual-based \cite{Verfurth2003}, equilibrated residual \cite{AinsworthOden2000,Braess2009,ErnVohralik2015,Li2025arxiv}, recovery-based \cite{Zienkiewicz1992b,ZhangNaga2005,Li2018SINUM,BankLi2019,Li2021JSCb}, dual weighted residual \cite{BeckerRannacher2001} and iteration-based \cite{Bank1993,Bank1996,LiShui2026,LiZikatanov2021CAMWA,LiZikatanov2025mcom} a posteriori error analysis. On simplicial meshes, a posteriori error estimation reaches a state of maturity, while its extension to polygonal FEs is nontrivial due to  geometric flexibility of polygonal elements and presence of non-polynomial basis functions. Readers are referred to \cite{BeiraoManzini2015,Cangiani2017,BerroneBorio2017,Beirao2019,DassiGedickeMascotto2022,ChaumontGedickeMascotto2025,Chen2014,Li2019,BertrandCarstensenGrassleTran2023,PorwalSingla2025} for a posteriori error estimates of virtual element methods, discontinuous FEs and hybrid high-order methods on polygonal meshes. However, to the best of our knowledge, a posteriori error analysis for polygonal FEs based on GBCs  is still an open question.

Wachspress GBC \cite{Wachspress1975} is the perhaps most famous conforming FEs on polygonal meshes. Without any stabilization, it makes use of explicit rational basis functions on meshes with convex polygonal elements. The main contribution of this work is to establish a residual-based a posteriori error estimate for the FE method based on Wachspress GBCs. 
Compared with polynomial FEs, the analysis contains two main difficulties on a mesh family satisfying suitable shape-regularity assumptions. 
The first one is to obtain uniform upper and lower bounds for the determinant and norm of the Jacobian matrix associated with polygon-to-polygon mappings. 
The second one is to prove inverse inequalities on polygonal elements with constants depending only on the mesh-regularity parameters.

Adaptive methods on polygonal meshes are less developed than their counterparts on simplicial or quadrilateral meshes. One of the difficulties is that polygonal refinement must simultaneously preserve mesh regularity, accommodate complex local topologies, and remain compatible with the approximation space. Existing approaches include subdivision strategies for polygonal cells \cite{Leon2014b}, centroidal Voronoi reconstruction \cite{NguyenXuan2017,Hoshina2018}, and adaptive virtual element methods \cite{Beirao2019,Wang2021,Du2022}. Related geometric results may also be found in \cite{Floater2016,Lai2016}. In this paper, we prove upper and lower bounds for a residual-type a posteriori estimator and combine this estimator with a local refinement strategy based on edge midpoints and element centroids, together with its extensions.

The rest of the paper is organized as follows. In Section~\ref{sect:prelim}, we introduce the model problem, the Wachspress-based polygonal finite element space, and the mesh regularity assumptions. In Section~\ref{sect:mapping}, we study the polygonal mapping induced by Wachspress coordinates and derive the geometric estimates needed in the analysis. Section~\ref{sect:residual} is devoted to proofs of reliability and efficiency of the residual-based a posteriori error estimator for Wachspress FE. Section~\ref{sect:numeric} presents numerical experiments to illustrate the effectiveness of the proposed estimator and adaptive algorithm.

\section{Preliminaries}\label{sect:prelim} Let $\Omega\subset\mathbb{R}^2$ be a bounded polygonal domain. We consider the Poisson problem
\begin{equation}\label{eq:poisson}
\left\{
\begin{aligned}
-\Delta u &= f \qquad &&\text{in }\Omega,\\
u &= 0 \qquad &&\text{on }\partial\Omega.
\end{aligned}
\right.
\end{equation}
Its weak formulation reads: find $u\in H_0^1(\Omega)$ such that
\begin{equation*}\label{eq:weak_poisson}
(\nabla u,\nabla v)=(f,v),\qquad \forall v\in H_0^1(\Omega).
\end{equation*}
We use $\|\cdot\|_{k,\Omega}$ and $|\cdot|_{k,\Omega}$ to denote the $H^k(\Omega)$ Sobolev norm and semi-norm, respectively.  
In particular, $|v|_{1,\Omega}=\|\nabla v\|_{0,\Omega}$ and $\|v\|_{0,\Omega}=\big(\int_\Omega v^2\big)^{1/2}$.

\subsection{Wachspress Barycentric Coordinates}\label{wachs_def}

Let $T\subset\mathbb{R}^2$ be a convex $n$-gon with vertices $\{\mathbf a_i\}_{i=1}^n$ ordered counterclockwise. Following \cite{Floater2015ref,Gillette2012}, we say $\{\lambda_{i,T}\}_{i=1}^n$ is a set of GBC on $T$ if 
\begin{equation*}
\lambda_{i,T}(\mathbf x)\geq 0, \quad\sum_{i=1}^n\lambda_{i,T}(\mathbf x)=1,\quad\sum_{i=1}^n\lambda_{i,T}(\mathbf x)\mathbf a_i=\mathbf x,\quad\forall \mathbf x\in T.    
\end{equation*}
We recall the definition of Wachspress GBC, which were introduced in \cite{Wachspress1975} and have been further investigated in, e.g., \cite{Wachspress1975,Loop1989,Warren1996,Meyer2002,Malsch2004,Warren2007}. Let $\mathbf n_i$ be the outward unit normal vector of the edge $e_i=[\mathbf a_i,\mathbf a_{i+1}]$ and set $M_i = \mathbf n_{i-1}\times \mathbf n_i$, where $\mathbf y\times \mathbf z:=y_1z_2-y_2z_1$ with $\mathbf y=(y_1,y_2)^\top$, $\mathbf z=(z_1,z_2)^\top$. For $\mathbf x\in T$, let $h_{i,T}(\mathbf x)$ be the distance from $\mathbf x$ to $e_i$, that is, 
\[
h_{i,T}(\mathbf x)=(\mathbf a_i-\mathbf x)\cdot\mathbf n_i=(\mathbf a_{i+1}-\mathbf x)\cdot\mathbf n_i.
\]

The Wachspress GBC is defined as
\begin{equation}\label{eq:wach_coordinate}
\lambda_{i,T}(\mathbf x):=\frac{w_{i,T}(\mathbf x)}{\sum_{j=1}^n w_{j,T}(\mathbf x)}=\frac{w_{i,T}(\mathbf x)}{W_T(\mathbf x)},\qquad i=1,\ldots,n.
\end{equation}
with the weight function given by 
\begin{equation}\label{eq:wach_weight_area}
w_{i,T}(\mathbf x):=M_i \prod_{j\neq i-1,i} h_{j,T}(\mathbf x).
\end{equation}

The Wachspress FE space of shape functions  on $T$ is \[
\Lambda_T:=\operatorname{span}\{\lambda_{1,T},\ldots,\lambda_{n,T}\}.
\]
Given a multi-index $\bm{\alpha}\in\mathbb{N}^2$, let    $|\boldsymbol\alpha|=\alpha_1+\alpha_2$. Let $\mathbb{P}_m={\rm span}\{\mathbf{x}^{\boldsymbol{\alpha}}:|\boldsymbol{\alpha}|\leq m\}$ denote the space of polynomials of degree $\leq m$. By definition, $\mathbb{P}_1\subset\Lambda_T$, $\lambda_{i,T}(\mathbf a_j)=\delta_{ij}$ for $1\le i,j\le n$, and the GBC-based interpolation $I_T$ preserves linear functions:
\[
p(\mathbf x)=I_Tp(\mathbf x):=\sum_{i=1}^n p(\mathbf a_i)\lambda_{i,T}(\mathbf x),\quad \forall\,p\in \mathbb P_1,\forall\,\mathbf x\in T. 
\]
For Wachspress GBC, restriction of $I_T$ to each edge $e\subset\partial T$ reduces to 1D Lagrange interpolation onto $\mathbb{P}_1(e)$, e.g., the linear polynomial space on $e$.

\subsection{Mesh Regularity}
We partition $\Omega$ into a set of $\mathcal{T}_h$ of convex polygonal elements. The Wachspress FE space is
\begin{equation}\label{eq:global_FEM_space}
V_h:=\{v_h\in C^0(\bar \Omega):\ v_h|_T\in \Lambda_T,\ \forall T\in\mathcal{T}_h,\, v_h|_{\partial \Omega} = 0\}.
\end{equation}
The Wachspress FE method is to find $u_h\in V_h$ such that
\begin{equation}\label{eq:WachspressFEM}
(\nabla u_h,\nabla v_h)=(f,v_h),\qquad \forall v_h\in V_h.
\end{equation}

For each $T\in\mathcal{T}_h$, let $\mathcal{V}_T=\{\mathbf a_i\}_{i=1}^{n(T)}$ be its set of vertices.  
Let $\theta_i$ be the interior angle of $T$ at $\mathbf a_i$, so that $0<\theta_i<\pi$. Let $h_T$ be the diameter of $T$, $\rho_T$ the diameter of the largest inscribed disk of $T$, and $h:=\max_{T\in\mathcal{T}_h} h_T$. Define 
\[
h_{*,T}:=\min_{1\le i\le n(T)}\Big(\min_{j\neq i-1,i} h_{j,T}(\mathbf a_i)\Big).
\]
Let $\mathcal{V}_h$ denote the set of mesh vertices. 
For each vertex $\mathbf a\in\mathcal{V}_h$, let $\Omega_{\mathbf a}:=\bigcup\{T\in\mathcal{T}_h:\mathbf a\in \mathcal{V}_T\}$
be the vertex patch surrounding $\mathbf a$.

We assume that the mesh family $\{\mathcal{T}_h\}$ satisfies the following shape-regularity conditions (cf.~\cite{Gillette2012,Chen2022,Tian2025}).
\begin{enumerate}
    \item[] \textbf{H1.} There exists a constant $C_1>0$ such that $h_T\le C_1\rho_T$ for all $T\in\mathcal{T}_h$.
    \item[] \textbf{H2.} There exists a constant $C_2>0$ such that $h_{*,T}\ge C_2 h_T$ for all $T\in\mathcal{T}_h$.
    \item[] \textbf{H3.} There exists a constant $\theta^*<\pi$ such that $\theta_i\le \theta^*$ for all $T\in\mathcal{T}_h$ and $1\le i\le n(T)$.
\end{enumerate}

Hypothesis \textbf{H1} implies a uniform lower bound on the interior angles:

\textbf{H4.} There exists a constant $\theta_*>0$ such that $0<\theta_*\le \theta_i$ for all $T\in\mathcal{T}_h$ and $1\le i\le n(T)$.

By convexity of each $T\in\mathcal{T}_h$ and \textbf{H4}, one further has:

\begin{enumerate}
    \item[] \textbf{H5.} There exists an integer $N>0$ such that $n(T)\leq N$ for all $T\in\mathcal{T}_h$.
    \item[] \textbf{H6.} There exists an integer $M>0$ such that $N(\Omega_{\mathbf a})\leq M$ for all $\mathbf a\in\mathcal{V}_h$, where $N(\Omega_{\mathbf a})$ denotes the number of elements in the patch $\Omega_{\mathbf a}$.
\end{enumerate}

Throughout the paper, all meshes are assumed to satisfy \textbf{H1}--\textbf{H6}. The symbols $c$, $C$ are generic positive constants depending only on \textbf{H1}--\textbf{H6}. By $C=C(\gamma)$ we mean the constant $C$ depends soly on $\gamma$. For two nonnegative quantities $A$ and $B$, we write $A\lesssim B$ if that $A\leq CB$, and write $A\eqsim B$ if both $A\lesssim B$ and $B\lesssim A$ hold.

For each $T\in\mathcal{T}_h$, it follows from \textbf{H3} and \textbf{H4} that there exists a constant $c_M>0$, depending only on $\theta^*$ and $\theta_*$, such that
\begin{equation}\label{eq:A_i_bound}
0<c_M\le M_i\le 1,\qquad i=1,\dots,n(T).
\end{equation}
By \textbf{H2}, for any $j$ and $k\notin\{j,j+1\}$,
\begin{equation}\label{eq:h_j_bound}
C_2h_T\leq h_{j,T}(\mathbf a_k)\leq h_T.
\end{equation}

For Wachspress GBC, the following derivative estimate developed in \cite{Tian2025} is key to our analysis.
\begin{lemma}\label{lem:wach_grad_bound}
For any vertex $\mathbf{a}_i\in\mathcal{V}_T$, and any $\boldsymbol\alpha\in\mathbb{N}^2$, the Wachspress barycentric coordinate $\lambda_{i,T}$ satisfies
\begin{equation*}
|D^{\boldsymbol\alpha}\lambda_{i,T}(\mathbf x)|
\le C h_T^{-|\boldsymbol\alpha|},
\qquad \forall\,\mathbf x\in T,
\end{equation*}
where the constant $C$ depends only on the mesh regularity and $|\boldsymbol\alpha|$.
\end{lemma}
In particular, $|\nabla\lambda_{i,T}|\lesssim h_T^{-1}$ and $|\nabla^2\lambda_{i,T}|\lesssim h_T^{-2}$.

\section{Mapping between Polygonal Elements}\label{sect:mapping}
We shall use homogeneity argument to derive a posteriori error bound. Therefore, it is necessary to analyze properties of a mapping from a reference polygon to the polygonal element $T\in\mathcal{T}_h$.

Let $T$ be a convex $n$-gon and let $\widehat T$ be a regular $n$-gon with $h_{\widehat T}=1$ and vertices $\{\hat{\mathbf a}_i\}_{i=1}^n$ ordered counterclockwise. Recall that $\lambda_{i,\hat{T}}$ is the Wachspress GBC at $\hat{\mathbf a}_i$, which satisfies 
\begin{equation}\label{eq:lambda_hat}
\|\lambda_{i,\hat{T}}\|_{L^\infty(\widehat T)}\leq1,\qquad\|\nabla\lambda_{i,\hat{T}}\|_{L^\infty(\widehat T)}\leq\widehat{C}_\lambda    
\end{equation}
for some constant $\widehat{C}_\lambda=\widehat{C}_\lambda(n)>0$.
Following \cite{Floater2010}, we define $F_T:\widehat T\to T$ by
\begin{equation}\label{def_polygon_map}
F_T(\hat{\mathbf x})=\sum_{i=1}^n \lambda_{i,\hat{T}}(\hat{\mathbf x})\,\mathbf a_i .
\end{equation}
Given a function $g$ on $T$, our analysis heavily relies on the pullback  $\hat{g}=g\circ F_T$ of $g$ defined
on the reference element $\widehat{T}$. We remark that the pullback  $\hat{\lambda}_i$ does not coincide with $\lambda_{i,\hat{T}}$. The pullback of $h_{i,T}\in\mathbb{P}_1$ is particularly important: 
\begin{equation}\label{eq:h_hat}
    \hat{h}_{i,T}(\hat{\mathbf x})
=
(I_Th_{i,T})(F(\hat{\mathbf x}))
=
\sum_{j=1}^n \lambda_{j,\hat{T}}(\hat{\mathbf x})\,h_{i,T}(\mathbf a_j).
\end{equation}

When no confusion arises, we may omit the subscript $T$, e.g., $\lambda_i=\lambda_{i,T}$, $h_i=h_{i,T}$, $F=F_T$. 

The mapping $F_T$ is rational and smooth in $\widehat T$. Let $J_F(\hat{\mathbf x})$ denote its Jacobian matrix and $\det(J_F(\hat{\mathbf x}))$ be its Jacobian determinant. To estimate $J_F$ and $\det(J_F)$, we recall the following results from \cite{Floater2010}.

\begin{lemma}\label{lem:D-scaling}
Let $a,b,c$ be differentiable functions on $\widehat T$, and let $\mu$ be a positive differentiable function on $\widehat T$. Define
\[
D(a,b,c):=
\det
\begin{pmatrix}
a & b & c\\
\partial_1 a & \partial_1 b & \partial_1 c\\
\partial_2 a & \partial_2 b & \partial_2 c
\end{pmatrix},
\]
where $\partial_i=\frac{\partial}{\partial x_i}$. Then, for any $\hat{\mathbf x}\in\widehat T$, it holds that
\[
D(\mu a,\mu b,\mu c)(\hat{\mathbf x})
=
\mu(\hat{\mathbf x})^3 D(a,b,c)(\hat{\mathbf x}).
\]
\end{lemma}

\begin{lemma}\label{lem:jacobian-expansion}
For any $\hat{\mathbf x}\in\widehat T$, it holds that 
\[
\det(J_F(\hat{\mathbf x}))
=
2\sum_{1\le i<j<k\le n}
D(\lambda_{i,\hat{T}},\lambda_{j,\hat{T}},\lambda_{k,\hat{T}})(\hat{\mathbf x})\,
A(\mathbf a_i,\mathbf a_j,\mathbf a_k),
\]
where $A(\mathbf u,\mathbf v,\mathbf w)
=
\frac12
\det
\begin{pmatrix}
1 & 1 & 1\\
u_1 & v_1 & w_1\\
u_2 & v_2 & w_2
\end{pmatrix}$.
\end{lemma}

\begin{lemma}\label{lem_polymap_inject}
For any $\hat{\mathbf x}\in\widehat T$, $\det(J_F(\hat{\mathbf x}))>0$ and $F_T:\widehat T\to T$ is injective.
\end{lemma}

Next, we show that \(F\) is surjective and thus bijective.

\begin{lemma}
If $\widehat T$ and $T$ are convex $n$-gons, then $F:\widehat T\to T$ is surjective.
\end{lemma}

\begin{proof}
Since each Wachspress GBC $\lambda_{i,\hat{T}}$ is continuous on $\widehat{T}$, the mapping $F$ is continuous. Moreover, for any $\hat{\mathbf x}\in\widehat T$, $F(\hat{\mathbf x})$ is a convex combination of $\mathbf a_1,\dots,\mathbf a_n$ and thus $F(\widehat{T})\subseteq T$.

For each edge $\hat{e}_i=[\hat{\mathbf a}_i,\hat{\mathbf a}_{i+1}]\subset\partial\widehat T$, all Wachspress GBCs vanish on $\hat{e}_i$ except $\lambda_{i,\hat{T}}$ and $\lambda_{i+1,\hat{T}}$. Thus $F(\hat{\mathbf x})
=
\lambda_{i,\hat{T}}(\hat{\mathbf x})\mathbf a_i+\hat\lambda_{i+1,\hat{T}}(\hat{\mathbf x})\mathbf a_{i+1}$ for all $\hat{\mathbf x}\in\hat{e}_i$. Therefore, $F([\hat{\mathbf a}_i,\hat{\mathbf a}_{i+1}])=[\mathbf a_i,\mathbf a_{i+1}]$, and  $F(\partial\widehat T)=\partial T$.

Assume that $F$ is not surjective. Then there exists $\mathbf y_0\in\operatorname{int}(T)\setminus F(\widehat T)$. Given $\mathbf{z}\in T\setminus\{\mathbf y_0\}$, let $\rho(\mathbf z)\in\partial T$ be the intersection of the ray $\{\mathbf y_0+t(\mathbf z-\mathbf y_0): t\geq0\}$ with $\partial T$. Thanks to the convexity of $T$, the map $\rho:\overline T\setminus\{\mathbf y_0\}\to\partial T$ is well-defined and continuous, with $\rho|_{\partial T}=\operatorname{id}_{\partial T}$.

Since $\tau=F|_{\partial\widehat T}:\partial\widehat T\to\partial T$ is a homeomorphism, we can define
\[
s:=\tau^{-1}\circ \rho\circ F:\widehat T\to\partial\widehat T,
\]
which is continuous with $s|_{\partial\widehat T}=\operatorname{id}_{\partial\widehat T}$. This contradicts the well-known fact that a (topological) disk cannot retract continuously onto its boundary. Therefore $F$ is surjective.
\end{proof}

To prove inverse inequalities on polygons, we need estimates for $J_F(\hat{\mathbf x})$. We use the following auxiliary results from \cite{Tian2025}.

\begin{lemma}\label{lem:wach_auxiliary}
For any $T\in\mathcal{T}_h$ and $\mathbf x\in T$, there are at most two indices $i, j\in\{1,2,\ldots,n\}$ satisfying $h_i(\mathbf x)<h_{*,T}/3$, $h_j(\mathbf x)<h_{*,T}/3$. If so, then the edges $e_i, e_j$ share a common vertex.
\end{lemma}


We present an estimate on $\det(J_F(\hat{\mathbf x}))$ in the next lemma.
\begin{lemma}\label{lem:Jacobian_det}
For any $\hat{\mathbf x}\in\widehat T$, it holds that
\[
\det(J_F(\hat{\mathbf x}))\eqsim h_T^2.
\]
\end{lemma}

\begin{proof}
We first prove the upper bound $\det(J_F(\hat{\mathbf x}))\lesssim h_T^2$. By Lemma \ref{lem:jacobian-expansion},
\begin{equation}\label{eq:det_expand_main}
\det(J_F(\hat{\mathbf x}))
=
2\sum_{1\le i<j<k\le n}
D(\lambda_{i,\hat{T}},\lambda_{j,\hat{T}},\lambda_{k,\hat{T}})(\hat{\mathbf x})\,
A(\mathbf a_i,\mathbf a_j,\mathbf a_k).
\end{equation}
Clearly \eqref{eq:lambda_hat} implies $|D(\lambda_{i,\hat{T}},\lambda_{j,\hat{T}},\lambda_{k,\hat{T}})(\hat{\mathbf x})|\leq 3\widehat{C}_\lambda^2$. Combining it with  $0\leq A(\mathbf a_i,\mathbf a_j,\mathbf a_k)\leq h_T^2$ and \eqref{eq:det_expand_main} gives
\[
\det(J_F(\hat{\mathbf x}))
\leq
6\binom{n}{3}\widehat{C}_\lambda^2h_T^2.
\]

We next prove the lower bound $\det(J_F(\hat{\mathbf x}))\gtrsim h_T^2$. In this proof, we adopt the notation $h_i(\hat{\mathbf{x}})=h_{i,\hat{T}}(\hat{\mathbf{x}})$, $w_i(\hat{\mathbf{x}})=w_{i,\hat{T}}(\hat{\mathbf{x}})$. Since $\widehat T$ is regular, all $\hat{\mathbf n}_{i-1}\times\hat{\mathbf n}_i$ are identical. Therefore, we redefine the Wachspress weights by
$w_i(\hat{\mathbf x})=\prod_{j\neq i-1,i}h_j(\hat{\mathbf x})$, 
$W_{\widehat T}(\hat{\mathbf x})=\sum_{i=1}^nw_i(\hat{\mathbf x})$ and introduce
\begin{align*}
&z_{i}(\hat{\mathbf x})=1/\big(h_{i-1}(\hat{\mathbf x})h_{i,\hat{T}}(\hat{\mathbf x})\big)=
w_i(\hat{\mathbf x})\big/\prod_{j=1}^n h_j(\hat{\mathbf x}),\\
&Z_{\widehat T}(\hat{\mathbf x})=\sum_{i=1}^n z_i(\hat{\mathbf x})=
W_{\widehat T}(\hat{\mathbf x})\big/\prod_{j=1}^n h_j(\hat{\mathbf x}),\\
&\lambda_{i,\hat{T}}(\hat{\mathbf x})=w_i(\hat{\mathbf x})/ W_{\widehat T}(\hat{\mathbf x})= z_i(\hat{\mathbf x})/Z_{\widehat T}(\hat{\mathbf x}).
\end{align*}

By Lemma \ref{lem:D-scaling}, $D(\lambda_{i,\hat{T}},\lambda_{j,\hat{T}},\lambda_{k,\hat{T}})
=
D(z_i, z_j, z_k) / Z_{\widehat T}^3$. By the decomposition formula in \cite{Floater2010}, we have 
\begin{align*}
&D(z_i, z_j, z_k)
=
z_i z_j z_k
\bigl(
T_{i,j,k}+T_{i,j-1,k-1}+T_{i-1,j,k-1}+T_{i-1,j-1,k}
\bigr),\\  
&T_{\alpha,\beta,\gamma} 
= 2h_\alpha(\hat{\mathbf{x}})^{-1} h_\beta(\hat{\mathbf{x}})^{-1} h_\gamma(\hat{\mathbf{x}})^{-1} 
d_{\alpha,\beta,\gamma}^{-1}\left|A(\bm b_\alpha,\bm b_\beta,\bm b_\gamma)\right|,
\end{align*}
where $\bm b_\alpha$, $\bm b_\beta$, $\bm b_\gamma$ are the intersections of the lines containing $e_\beta$ and $e_\gamma$, $e_\gamma$ and $e_\alpha$, and $e_\alpha$ and $e_\beta$, respectively, and $d_{\alpha,\beta,\gamma}$ denotes the diameter of the circumcircle of the triangle with vertices $\bm b_\alpha$, $\bm b_\beta$, $\bm b_\gamma$, see Figure~\ref{fig:T_intersect_b}.

It follows from Lemma \ref{lem:wach_auxiliary} that
\begin{equation}\label{eq:h_split_case}
 h_i(\hat{\mathbf x})\geq h_{*,\widehat T}/3,
\qquad
i\neq m,m+1,
\end{equation}
for some index $m$.
To estimate $D(\lambda_{m,\hat{T}},\lambda_{m+1,\hat{T}},\lambda_{m+2,\hat{T}})$, we calculate
\begin{align*}
    &z_m z_{m+1} z_{m+2}T_{m,m+1,m+2}/Z_{\widehat T}^3 \\
    &=
    \frac{w_{m}(\hat{\mathbf{x}})w_{m+1}(\hat{\mathbf{x}}) w_{m+2}(\hat{\mathbf{x}})}{W_{\widehat T}^3}
    \frac{1}{ h_{m}(\hat{\mathbf x}) h_{m+1}(\hat{\mathbf x}) h_{m+2}(\hat{\mathbf{x}})}
    \frac{2|A(\bm b_m,\bm b_{m+1},\bm b_{m+2})|}{d_{m,m+1,m+2}}.
\end{align*}

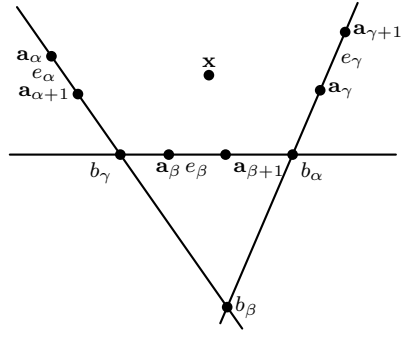
\begin{figure}[h]
    \centering
    \begin{tikzpicture}[scale=1, line join=round, line cap=round]

\coordinate (aalpha)    at (-0.45,0.00);
\coordinate (aalphaone) at (0.30,0.00);

\coordinate (abeta)     at (-2.00,1.30);
\coordinate (abetaone)  at (-1.65,0.80);

\coordinate (agamma)    at (1.55,0.85);
\coordinate (agammaone) at (1.88,1.62);

\coordinate (x)         at (0.08,1.05);

\path[name path global=ealpha] (-2.55,0) -- (2.55,0);

\path[name path global=ebeta] 
    ($(abeta)!-1.3!(abetaone)$) -- ($(abeta)!7.2!(abetaone)$);

\path[name path global=egamma] 
    ($(agamma)!-4.0!(agammaone)$) -- ($(agamma)!1.5!(agammaone)$);

\path[name intersections={of=ealpha and ebeta, by=bgamma}];
\path[name intersections={of=ealpha and egamma, by=bbeta}];
\path[name intersections={of=ebeta and egamma, by=balpha}];

\draw[thick] (-2.55,0) -- (2.55,0);
\draw[thick] ($(abeta)!-1.3!(abetaone)$) -- ($(abeta)!7.2!(abetaone)$);
\draw[thick] ($(agamma)!-4.0!(agammaone)$) -- ($(agamma)!1.5!(agammaone)$);

\foreach \p in {aalpha,aalphaone,abeta,abetaone,agamma,agammaone,balpha,bbeta,bgamma,x}
  \fill (\p) circle (2pt);

\node[below] at (aalpha)    {$\mathbf a_\beta$};
\node[below right] at (aalphaone) {$\mathbf a_{\beta+1}$};

\node[left]  at (abeta)     {$\mathbf a_\alpha$};
\node[left]  at (abetaone)  {$\mathbf a_{\alpha+1}$};

\node[right] at (agamma)    {$\mathbf a_\gamma$};
\node[right] at (agammaone) {$\mathbf a_{\gamma+1}$};

\node[below left]  at (bgamma) {$b_\gamma$};
\node[below right] at (bbeta)  {$b_\alpha$};
\node[right] at (balpha) {$b_\beta$};

\node[above] at (x) {$\mathbf x$};

\node[below] at ($(aalpha)!0.5!(aalphaone)$) {$e_\beta$};
\node[left]  at ($(abeta)!0.5!(abetaone)$) {$e_\alpha$};
\node[right] at ($(agamma)!0.5!(agammaone)$) {$e_\gamma$};

\end{tikzpicture}
    \caption{Illustration of the definition of $\bm b_\alpha$, $\bm b_\beta$, and $\bm b_\gamma$.}
    \label{fig:T_intersect_b}
\end{figure}

Since $\widehat T$ is regular, there exists $c_1=c_1(n)>0$ such that
\[
\frac{2|A(\bm b_m,\bm b_{m+1},\bm b_{m+2})|}
{d_{m,m+1,m+2}}
\ge c_1 .
\]
By \eqref{eq:h_split_case} and the fact that
$W_{\widehat T}(\hat{\mathbf x})\le n$, we obtain
\begin{align*}
&D(\lambda_{m,\hat{T}},\lambda_{m+1,\hat{T}},\lambda_{m+2,\hat{T}})(\hat{\mathbf{x}})=D(z_{m}, z_{m+1}, z_{m+2}) / Z_{\widehat T}^3\\
&\geq
n^{-3}
c_1
\frac{
w_{m}(\hat{\mathbf x})w_{m+1}(\hat{\mathbf x}) w_{m+2}(\hat{\mathbf x})}
{
h_{m}(\hat{\mathbf x})
h_{m+1}(\hat{\mathbf x})
h_{m+2}(\hat{\mathbf x})
}
\geq
n^{-3}c_1
\left(\frac{h_{*,\widehat T}}{3}\right)^{3n-9}
=:c_2 .
\end{align*}

By \textbf{H2}, $|\mathbf a_{m+1}-\mathbf a_m|\ge h_{*,T}$ and $h_{*,T}\ge C_2h_T$. Since $\mathbf a_{m+2}$ does not lie on the line through $\mathbf a_m$ and $\mathbf a_{m+1}$, its distance to this line is at least $h_{*,T}$. Hence
\[
A(\mathbf a_m,\mathbf a_{m+1},\mathbf a_{m+2})
\ge
\frac12|\mathbf a_{m+1}-\mathbf a_m|\,h_{*,T}
\ge \frac{1}{2}C_2^2h_T^2.
\]
Using \eqref{eq:det_expand_main} and the nonnegativity of all terms, we obtain
\[
\det(J_F(\hat{\mathbf x}))
\ge
2D(\lambda_{m,\hat{T}},\lambda_{m+1,\hat{T}},\lambda_{m+2,\hat{T}})(\hat{\mathbf x})
A(\mathbf a_m,\mathbf a_{m+1},\mathbf a_{m+2})
\ge c_2C_2^2h_T^2.
\]
This completes the proof.\qed
\end{proof}

Meanwhile, we are able to estimate the spectral matrix norms $\|J_F(\hat{\mathbf x})\|_2$ and $\|J_F^{-1}(\hat{\mathbf x})\|_2$.
\begin{lemma}\label{lem:Jacobian_norm}
For any $\hat{\mathbf x}\in\widehat T$, it holds that
\begin{align*}
    \|J_F(\hat{\mathbf x})\|_2&\lesssim h_T,\\
\|J_F^{-1}(\hat{\mathbf x})\|_2&\lesssim h_T^{-1}.
\end{align*}
\end{lemma}

\begin{proof}
We first prove $\|J_F(\hat{\mathbf x})\|_2\lesssim h_T$. Let $F(\hat{\mathbf x})=(F_1(\hat{\mathbf x}),F_2(\hat{\mathbf x}))^\top$, $\mathbf{a}_i = (a_i^1, a_i^2)^\top$,  
$F_\ell(\hat{\mathbf x})=\sum_{i=1}^n\lambda_{i,\hat{T}}(\hat{\mathbf x})a_i^\ell,
\ \ell=1,2$. Then
\[
J_F(\hat{\mathbf x})
=
\begin{pmatrix}
\sum_{i=1}^n(\partial_1\lambda_{i,\hat{T}})a_i^1 &
\sum_{i=1}^n(\partial_2\lambda_{i,\hat{T}})a_i^1\\[0.4em]
\sum_{i=1}^n(\partial_1\lambda_{i,\hat{T}})a_i^2 &
\sum_{i=1}^n(\partial_2\lambda_{i,\hat{T}})a_i^2
\end{pmatrix}.
\]
Since $\sum_{i=1}^n\lambda_{i,\hat{T}}(\hat{\mathbf{x}})=1$, we have $\sum_{i=1}^n\partial_r\lambda_{i,\hat{T}}(\hat{\mathbf x})=0$ for $r=1,2$. Thus \[
\sum_{i=1}^n(\partial_r\lambda_{i,\hat{T}})a_i^\ell
=
\sum_{i=2}^n(\partial_r\lambda_{i,\hat{T}})(a_i^\ell-a_1^\ell),
\]
for $r,\ell\in\{1,2\}$, and hence
\[
J_F(\hat{\mathbf x})
=
\begin{pmatrix}
\sum_{i=2}^n(\partial_1\lambda_{i,\hat{T}})(a_i^1-a_1^1) &
\sum_{i=2}^n(\partial_2\lambda_{i,\hat{T}})(a_i^1-a_1^1)\\[0.4em]
\sum_{i=2}^n(\partial_1\lambda_{i,\hat{T}})(a_i^2-a_1^2) &
\sum_{i=2}^n(\partial_2\lambda_{i,\hat{T}})(a_i^2-a_1^2)
\end{pmatrix}.
\]
Using \eqref{eq:lambda_hat} and $|a_i^\ell-a_1^\ell|\leq |\mathbf a_i-\mathbf a_1|\leq h_T$, we have
\begin{align*}
    \left|
\sum_{i=2}^n(\partial_r\lambda_{i,\hat{T}})(a_i^\ell-a_1^\ell)
\right|
&\leq
\left(\sum_{i=2}^n|\partial_r\lambda_{i,\hat{T}}|^2\right)^{1/2}
\left(\sum_{i=2}^n|a_i^\ell-a_1^\ell|^2\right)^{1/2}\\
&\leq (n-1)\widehat{C}_\lambda h_T,
\end{align*}
which implies that $\|J_F(\hat{\mathbf x})\|_2\lesssim h_T$.

Combining $\det(J_F(\hat{\mathbf x}))\eqsim h_T^2$ (Lemma \ref{lem:Jacobian_det}) with $\|J_F(\hat{\mathbf x})\|_2\lesssim h_T$ and
\[
\|J_F^{-1}(\hat{\mathbf x})\|_2
=
\frac{\|J_F(\hat{\mathbf x})\|_2}{\det(J_F(\hat{\mathbf x}))},
\]
we verify $\|J_F^{-1}(\hat{\mathbf x})\|_2\lesssim h_T^{-1}$ and complete the proof. \qed
\end{proof}

In principle, trace inequalities on a polygonal element $T$ can be proved using the transformation $F_T$ and Lemmas \ref{lem:Jacobian_norm} and \ref{lem:Jacobian_det}. A more direct proof without homogeneity argument can be found in, e.g., \cite{ChaumontFrelet2025}. 
\begin{lemma}[Trace inequality]\label{lem_trace}
Let $T$ be a polygon in $\mathcal{T}_h$, and let $e\subset\partial T$ be an edge. For any $v\in H^1(T)$, it holds that 
\[
\|v\|_{0,e}
\lesssim
h_T^{-1/2}\|v\|_{0,T}
+
h_T^{1/2}|v|_{1,T}.
\]
\end{lemma}

\section{Residual-based A Posterior Error Estimate}\label{sect:residual}
In this section, we derive a residual-based error estimator for the Wachspress finite element method \eqref{eq:WachspressFEM}.

\subsection{A Posteriori Error Upper Bound}
A crucial ingredient for establishing a posteriori error upper bound is the Scott-Zhang interpolation on polygonal meshes. By \textbf{H6}, the number of elements in each vertex patch $\Omega_{\mathbf a}$ is uniformly bounded by $M$, and hence
\begin{equation*}
h_T\le h_{\Omega_T}\leq 2^{\left[M/2\right]}h_T,
\end{equation*}
where $\Omega_T=\cup_{\mathbf a\in\mathcal V_T}\Omega_{\mathbf a}$ and $h_{\Omega_T}=\operatorname{diam}(\Omega_T)$.

Given a vertex $\mathbf{a}\in\mathcal{V}_h$, let $\lambda_{\mathbf{a}}\in V_h$ be the generalized hat function with $\lambda_{\mathbf{a}}(\mathbf{a}^\prime)=\delta_{\mathbf{a},\mathbf{a}^\prime}$ for any $\mathbf{a}^\prime\in\mathcal{V}_h$. The support of $\lambda_{\mathbf{a}}$ is  $\Omega_{\mathbf a}$. The restriction of $\lambda_{\mathbf{a}}$ to each element is a Wachspress GBC. Let $e_{\mathbf{a}}$ is an edge having $\mathbf{a}$ as an endpoint, and $e_{\mathbf{a}}$ is chosen as a boundary edge if $\mathbf{a}\in\partial\Omega$. The Scott-Zhang interpolation $\Pi_h$ is 
\begin{equation}\label{scottzhang_def}
\Pi_hv=\sum_{\mathbf{a}\in\mathcal{V}_h} \left(\int_{e_\mathbf{a}}v\psi_{\mathbf{a}}\right)\lambda_{\mathbf{a}}=\sum_{\mathbf{a}\in\mathcal{V}_h} L_{\mathbf{a}}(v)\lambda_{\mathbf{a}},
\end{equation}
where $\psi_{\mathbf a}\in\mathbb{P}_1(e_{\mathbf a})$ is a weight function determined by
$\int_{E_{\mathbf a}}\lambda_{\mathbf{a}^\prime}\psi_{\mathbf a}=\delta_{\mathbf{a},\mathbf{a}^\prime}$. 

The operator $\Pi_h$ satisfies the following estimates.

\begin{lemma}\label{lem_intplat}
For any $T\in\mathcal{T}_h$, $v\in H^1(\Omega_T)$ and $p\in\mathbb P_1(\Omega_T)$, it holds that
\begin{subequations}
    \begin{align}
\Pi_hp&=p~\text{ on }T,\label{eq:linear_preserve}\\
\|v-\Pi_hv\|_{0,T}&\lesssim h_T|v|_{1,\Omega_T},\label{eq:v-Pihv}\\
\|v-\Pi_hv\|_{0,\partial T}&\lesssim h_T^{1/2}|v|_{1,\Omega_T},\label{eq:v-Pihv_trace}\\
|\Pi_hv|_{1,T}&\lesssim |v|_{1,\Omega_T}.\label{eq:Pihv_H1}
\end{align}
\end{subequations}
\end{lemma}

\begin{proof}
Any $p\in\mathbb{P}_1(\Omega_T)$ can be decomposed as $p=\sum_{\mathbf{a}\in T} c_{\mathbf a}\lambda_{\mathbf a}$. The definition of $\psi_{\mathbf a}$ ensures $L_{\mathbf{a}^\prime}(\lambda_{\mathbf{a}})=\delta_{\mathbf{a},\mathbf{a}^\prime}$ and $\Pi_h\lambda_{\mathbf{a}}=\lambda_{\mathbf{a}}$, which implies  \eqref{eq:linear_preserve}.

By the Bramble-Hilbert lemma, we can pick
a $w\in\mathbb P_0(\Omega_T)$ satisfying 
\begin{equation}\label{eq:BH_star}
\|v-w\|_{0,\Omega_T}\lesssim h_{\Omega_T}|v|_{1,\Omega_T}.
\end{equation}
Using $|\psi_{\mathbf a}|\lesssim h_{e_{\mathbf a}}^{-1}$, Lemma \ref{lem_trace} and \eqref{eq:BH_star}, we have 
\begin{equation*}
    \begin{aligned}
        |L_{\mathbf a}(v-w)|
        &\leq\|\psi_{\mathbf a}\|_{0,e_{\mathbf a}}\|v-w\|_{0,e_{\mathbf a}}\lesssim h_{e_{\mathbf a}}^{-\frac{1}{2}}\|v-w\|_{0,e_{\mathbf a}}\\
        &
        \lesssim h_T^{-1}\|v-w\|_{0,\Omega_T}+ |v - w|_{1, \Omega_T} \lesssim |v|_{1,\Omega_T}.
    \end{aligned}
\end{equation*}
Combining the above bound for $L_{\mathbf{a}}(v-w)$ with $\|\lambda_{\mathbf a}\|_{0,T}\lesssim h_T$ and $\|\nabla\lambda_{\mathbf a}\|_{0,T}\lesssim1$, we arrive at
\begin{equation}\label{eq:Piv-w}
\begin{aligned}
\|\Pi_h(v-w)\|_{0,T}&\lesssim h_T|v|_{1,\Omega_T},\\
|\Pi_h(v-w)|_{1,T}&\lesssim |v|_{1,\Omega_T}.    
\end{aligned}
\end{equation}
Using $\Pi_hw=w$ and \eqref{eq:Piv-w}, we obtain
\begin{align*}
    \|v-\Pi_h v\|_{0,T}&=\|v-w-\Pi_h (v-w)\|_{0,T}\\
    &\leq\|v-w\|_{0,T}+\|\Pi_h (v-w)\|_{0,T}\\
    &\lesssim h_T|v|_{\Omega_T}
\end{align*}
and verified \eqref{eq:v-Pihv}. The bound \eqref{eq:Pihv_H1} is proved in a similar way. The bound \eqref{eq:v-Pihv_trace} follows from \eqref{eq:v-Pihv}, \eqref{eq:Pihv_H1} and Lemma \ref{lem_trace}. \qed
\end{proof}

Let $\mathcal E_h$ be the set of interior edges. Each $e\in\mathcal{E}_h$ is assigned with a unit normal vector $\mathbf{n}_e$. Let $\llbracket\frac{\partial u_h}{\partial\mathbf n_e}\rrbracket$ denote the jump of the normal derivative of $u_h$ across $e$. We define the residual-based error indicator as follows. 
\begin{equation}\label{eq_eta_total}
    \eta_T^2
    =
    h_T^2\|f+\Delta u_h\|_{0,T}^2
    +\frac12\sum_{e\in\mathcal E_h,\, e\subset\partial T} h_e\left\|\jump*{\frac{\partial u_h}{\partial\mathbf n_e}}\right\|_{0,e}^2.
\end{equation}

\begin{theorem}[Reliability]\label{thm:reliability_res}
For the Wachspress FE method \eqref{eq:WachspressFEM}, 
\[
|u-u_h|_{1,\Omega}\lesssim \eta_h:=\left(\sum_{T\in\mathcal T_h}\eta_T^2\right)^\frac{1}{2},
\]
\end{theorem}

\begin{proof}
For any $v\in H_0^1(\Omega)$, By Galerkin orthogonality and element-wise integration by parts, we have
\begin{equation}\label{eq:residual_representation}
\begin{aligned}
&(\nabla(u-u_h),\nabla v)
=
(\nabla(u-u_h),\nabla (v-\Pi_h v))\\
&=
\sum_{T\in\mathcal T_h}(f+\Delta u_h,v-\Pi_h v)_T
-
\sum_{e\in\mathcal E_h}\int_e \jump*{\frac{\partial u_h}{\partial\mathbf n_e}}\,(v-\Pi_h v). 
\end{aligned}
\end{equation}
Combining \eqref{eq:residual_representation} with a Cauchy-Schwarz inequality and Lemma \ref{lem_intplat}, we obtain
\[
(\nabla(u-u_h),\nabla v)
\lesssim
\left(
\sum_{T\in\mathcal T_h} h_T^2\|f+\Delta u_h\|_{0,T}^2
+
\sum_{e\in\mathcal E_h} h_T\left\|\jump*{\frac{\partial u_h}{\partial\mathbf n_e}}\right\|_{0,e}^2
\right)^{1/2}
|v|_{1,\Omega}.
\]
Taking $v=u-u_h$ then yields $|u-u_h|_{1,\Omega}\lesssim \eta_h.$ \qed
\end{proof}

\subsection{A Posteriori Error Lower Bound}
This section is devoted to the proof of the a posteriori error lower bound. 
\begin{theorem}[Efficiency]\label{thm:efficiency}
Let $f_T=|T|^{-1}\int_T f$. It holds that
\[
\eta_h^2 \lesssim |u-u_h|_{1,\Omega}^2+\sum_{T\in\mathcal T_h} h_T^2\|f-f_T\|_{0,T}^2.
\]
\end{theorem}

For Wachspress FE, the mapping $F_T$ is nonlinear, and the pullback of the relevant space is no longer a polynomial space. To quantify the effect of pullbacks, we need the results in Section \ref{sect:mapping} and the following estimate on the Wachspress total weight $W=\sum_{i=1}^nw_i$.
\begin{lemma}\label{lem:What}
For $T\in\mathcal{T}_h$ and $\hat{\mathbf x}\in\widehat T$, it holds that
\begin{subequations}
    \begin{align}
        c_M (C_2\lambda_{i,\hat{T}}(\hat{\mathbf x})h_T)^{n-2}\le \hat w_i(\hat{\mathbf{x}})\le h_T^{n-2}, \label{eq:w_ihat}\\
        c_0 h_T^{n-2}\leq \widehat W(\hat{\mathbf x})\leq n h_T^{n-2}, \label{eq:What}\\
        |\nabla \widehat W(\hat{\mathbf x})|\leq n^3 \widehat{C}_\lambda h_T^{n-2}, \label{eq:gradWhat}
    \end{align}
\end{subequations}
where the constant $c_0>0$ depends only on $C_2$, $N$, $\theta^*$, $\theta_*$.
\end{lemma}
\begin{proof}
Using the expression \eqref{eq:h_hat}, we have
\begin{equation}\label{eq:w_hat}
\hat w_i(\hat{\mathbf x})= M_i\prod_{j\neq i-1,i}\hat{h}_j= M_i\prod_{j\neq i-1,i}\sum_{k=1}^{n}\lambda_{k,\hat{T}}(\hat{\mathbf x})h_{j}(\mathbf a_k).
\end{equation}
It then follows from \eqref{eq:A_i_bound}, \eqref{eq:h_j_bound} and \eqref{eq:w_hat} that
\begin{align*}
    c_M (C_2\lambda_{i,\hat{T}}(\hat{\mathbf x})h_T)^{n-2}
    \le c_M\prod_{j\neq i-1,i}\lambda_{i,\hat{T}}(\hat{\mathbf x})h_{j}(\mathbf a_i)
    \le
    \hat w_i(\hat{\mathbf x}) 
    \le h_T^{n-2}, \\
    c_M(C_2h_T)^{n-2}\sum_{i=1}^n\lambda_{i,\hat{T}}^{n-2}(\hat{\mathbf x})
    \le
    \widehat W(\hat{\mathbf x})
    \le
    nh_T^{\,n-2},
\end{align*}
which imply \eqref{eq:w_ihat} and \eqref{eq:What} with $c_0:=c_M C_2^{\,n-2}n^{3-n}$ and the Jensen's inequality $\sum_{i=1}^n \lambda_{i,\hat{T}}^{\,n-2}(\hat{\mathbf{x}}) \ge n^{3-n}$. By \eqref{eq:A_i_bound}, \eqref{eq:h_j_bound} and Lemma \ref{lem:wach_grad_bound}, we obtain 
\begin{align*}
     |\nabla \widehat W(\hat{\mathbf x})| 
     \leq& 
     \sum_{i=1}^n \sum_{\substack{l=1 \\ l\neq i-1,i}}^{n}
     \Big(\prod_{j\neq i-1,i,l} \sum_{k=1}^{n}\lambda_{k,\hat{T}}(\hat{\mathbf x})h_{j}(\mathbf a_k) \Big)
     \Big( \sum_{k=1}^{n} |\nabla\lambda_{k,\hat{T}}(\hat{\mathbf x})|h_{l}(\mathbf a_k) \Big)\\
     \leq&
     \sum_{i=1}^n \sum_{\substack{l=1 \\ l\neq i-1,i}}^{n} n\widehat{C}_\lambda h_T^{n-2}\le n^3\widehat{C}_{\lambda}h_T^{n-2}.
\end{align*}
This proves \eqref{eq:gradWhat}.\qed
\end{proof}

For $T\in\mathcal{T}_h$ and the reference element $\widehat T$, we introduce 
\begin{align*}
V_T&:=\operatorname{span}\{\lambda_1,\dots,\lambda_n,\Delta\lambda_1,\dots,\Delta\lambda_n\},\quad
\widehat V_T:=\{v\circ F_T:\ v\in V_T\}\\
\widehat{U}_r
&:=
\operatorname{span}\Big\{\lambda_{1,\hat{T}}^{\alpha_1}\cdots \lambda_{n,\hat{T}}^{\alpha_n}
:\ (\alpha_1,\ldots,\alpha_n)\in\mathbb{N}^n,\ \sum_{i=1}^n\alpha_i\leq r
\Big\}.
\end{align*}

In the following, we present inverse inequalities on polygons.
\begin{lemma}\label{lem:full_inverse_reference}
Let $\hat{b}>0$ a.e. on $\widehat{T}$ be a weight function. Then
\begin{equation*}
|\hat v|_{1,\widehat T}\leq C_{\rm inv}\|\hat{b}^{1/2}\hat{v}\|_{0,\widehat T},
\qquad \forall\,\hat v\in \widehat{V}_T.
\end{equation*}
where $C_{\rm inv}>0$ depends only on $C_2$, $N$, $\theta^*$, $\theta_*$ and $\hat{b}$.
\end{lemma}

\begin{proof}
The expression \eqref{eq:w_hat} implies $\hat{w}_i\in\widehat{U}_{n-2}$ and thus $\widehat W\in\widehat{U}_{n-2}$. Recall $\lambda_i=w_i/W$ and thus $\Delta\lambda_i=z_i/W^3$ with 
\[
z_i
=
W^2\Delta w_i-w_iW\Delta W-2W\nabla w_i\cdot\nabla W+2w_i|\nabla W|^2.
\]
Therefore, $\hat{z}_i\in\widehat{U}_{3n-8}$. Any $\hat v\in\widehat{V}_T$ can be written as
\[
\hat v
=
\sum_{i=1}^n \alpha_i\frac{\widehat w_i}{\widehat W}
+\sum_{i=1}^n \beta_i\frac{\hat z_i}{\widehat W^3}
=
\frac{\sum_{i=1}^n \alpha_i\,\hat{w}_i\widehat{W}^2+\sum_{i=1}^n \beta_i\,\hat z_i}{\widehat{W}^3}.
\]
Thus $\hat v=\hat p/\widehat{W}^3$ with $\hat p:=\sum_{i=1}^n \alpha_i\,\hat w_i\widehat W^2+\sum_{i=1}^n \beta_i\,\hat z_i\in\widehat{U}_{3n-6}$. 

By norm equivalence of the finite-dimensional space $\widehat{U}_{3n-6}$, there exists $C_{\mathrm{fd}}>0$, depending only on $n$, such that
\begin{equation}\label{eq:polygon_finite_fanshu}
\|\hat p\|_{1,\widehat T}\le C_{\mathrm{fd}}\|\hat{b}^{1/2}\hat{p}\|_{0,\widehat T},
\qquad \forall\,\hat p\in\widehat{U}_{3n-6}.
\end{equation}
By Lemma \ref{lem:What}, we have $\|\widehat W^{-3}\|_{L^\infty(\widehat T)}\le c_0^{-3}h_T^{-3n+6}$ and  $\|\nabla(\widehat W^{-3})\|_{L^\infty(\widehat T)}
\le
3c_0^{-4}n^3\widehat{C}_\lambda h_T^{-3n+6}$. Combining Lemma \ref{lem:What} with \eqref{eq:polygon_finite_fanshu} yields
\begin{align*}
    \|\nabla\hat v\|_{0,\widehat T}
    =& \|\widehat W^{-3}\nabla\hat p+\hat p\,\nabla(\widehat W^{-3})\|_{0,\widehat T}
    \le
    C_T h_T^{-3n+6}\|\hat{b}^{1/2}\hat p\|_{0,\widehat T} \\
    =& C_T h_T^{-3n+6}\|\widehat W^3\hat{b}^{1/2}\hat v\|_{0,\widehat T}
    \le n^3C_T\|\hat{b}^{1/2}\hat{v}\|_{0,\widehat T}.
\end{align*}
with $C_T:=(c_0^{-3}+3c_0^{-4}n^3\widehat{C}_\lambda)(1+C_{\mathrm{fd}})$. This finishes the proof. \qed
\end{proof}

\begin{theorem}[Inverse Estimate]\label{thm:inverse}
For any $T\in\mathcal T_h$ and $v_h\in V_T$,
\[
|v_h|_{1,T}\lesssim h_T^{-1}\|v_h\|_{0,T}.
\]
\end{theorem}

\begin{proof}
The proof directly follows from Lemmas \ref{lem:Jacobian_norm}, \ref{lem:Jacobian_det} and \ref{lem:full_inverse_reference}. \qed
\end{proof}

For an interior edge $e\in\mathcal E_h$, let $T^+$ and $T^-$ be the two convex polygonal elements sharing $e$. We use a superscript $+$ or $-$ to indicate a quantity is defined either on $T^+$ or $T^-$. For example, $F^\pm:\widehat T\to T^\pm$ is the homogeneity mapping from $\widehat{T}$ to $T^\pm$, and $w_i^{\pm}$ is the $i$-th Wachspress GBC on $T^\pm$.

The edge $e$ is connected to the fixed interval $\hat e=[0,1]$ via 
\[
F_e:\hat e\to e,\qquad F_e(t)=(1-t)\mathbf a+t\mathbf b,
\]
where $\mathbf a,\mathbf b$ are the endpoints of $e$. For a function $g$ defined on $e$, its pullback to $\hat{e}$ is $\hat g=g\circ F_e$. We define space
\[
V_e=\mathbf n_e\cdot {\rm span}\big\{\nabla\lambda_{1,T^+},\ldots,\nabla\lambda_{n^+,T^+},\nabla\lambda_{1,T^-},\ldots,\nabla\lambda_{n^-,T^-}\big\}\big|_e.
\]

\begin{lemma}\label{lem:ref_edge_inverse}
Let $\hat{b}>0$ a.e. on $e$ be a weight. 
For any $v\in V_e$, it holds that
\begin{equation*}
        \|\hat v\|_{1,\hat e}+\|\hat v\|_{\infty,\hat e}
        \leq C_{\rm inv,e}\|\hat b^{1/2}\hat v\|_{0,\hat e}
\end{equation*}
where $C_{\rm inv,e}>0$ depend only on $C_2$, $N$, $\theta^*$, $\theta_*$. 
\end{lemma}

\begin{proof}
Using $\mathbf n_e\cdot\nabla\lambda_i^\pm=\mathbf n_e\cdot(W^\pm\nabla w_i^\pm-w_i^\pm\nabla W^\pm)/(W^\pm)^2:=r_i^\pm/(W^\pm)^2$, we can write $v\in V_e$ as
\begin{equation}\label{eq:edge_v_pullback_rep}
\begin{aligned}
    \hat v
&= v\circ F_e=
\sum_{i=1}^{n_+}\alpha_i^+\frac{\hat r_i^+}{(\widehat W^+)^2}
+
\sum_{i=1}^{n_-}\alpha_i^-\frac{\hat r_i^-}{(\widehat W^-)^2}\\
&=
\frac{
(\widehat W^-)^2\sum_{i=1}^{n_+}\alpha_i^+\widehat r_i^+
+
(\widehat W^+)^2\sum_{i=1}^{n_-}\alpha_i^-\hat r_i^-
}{
(\widehat W^+)^2(\widehat W^-)^2
}\\
&=\frac{\hat p_e}{(\widehat W^+)^2(\widehat W^-)^2}.
\end{aligned}
\end{equation}
Since $F_e$ is affine, each $h_j^\pm\circ F_e$ is a linear polynomial on $\hat e$. Therefore, $\widehat W^\pm\in \mathbb P_{n_\pm-2}(\hat e)$, $\widehat r_i^\pm\in \mathbb P_{2n_\pm-5}(\hat e)$, and $\hat{p}_e\in \mathbb P_{4N-9}(\hat e)$.

By norm equivalence on $\mathbb{P}_{4N-9}(\hat e)$,  
\begin{equation}\label{eq:phat}
\|\hat p_e\|_{1,\hat e}+\|\hat p_e\|_{\infty,\hat e}\leq \widehat C\|\hat b^{1/2}\hat p_e\|_{0,\hat e},
\end{equation}
where $\widehat C>0$ depends only on $N$.
Using Lemma \ref{lem:What} on $T^\pm$, we have
\begin{equation}
        \label{eq:edge_W_bounds}
    c_{\pm}h_e^{\,n_\pm-2}
        \leq \widehat W^\pm(t)
        \leq C_{\pm}h_e^{\,n_\pm-2},
        \qquad \forall t\in\hat{e}.
\end{equation}
It follows from \eqref{eq:edge_v_pullback_rep}, \eqref{eq:phat} and \eqref{eq:edge_W_bounds} that
\begin{align*}
    \|\hat v\|_{0,\hat e}
    \le& c_+^{-2}c_-^{-2}\|\hat p_e\|_{0, \hat e}
    \le c_{+}^{-2}c_{-}^{-2}\widehat C\|\hat{b}^{1/2}\hat p_e\|_{0,\hat e}
    \\
    \leq& c_{+}^{-2}c_{-}^{-2}C_{+}^2C_{-}^2\widehat C\|\hat{b}^{1/2}\hat v\|_{0,\hat e}. 
\end{align*}
Similary, it follows from \eqref{eq:edge_W_bounds} and \eqref{eq:phat} that
\begin{align*}
    \|\hat v\|_{L^\infty(\hat e)}
    \le& 
    (c_{+}c_{-})^{-2}h_e^{-2n_+-2n_-+8}\|\hat p_e\|_{L^\infty(\hat e)}\\
    \le&
    (c_{+}c_{-})^{-2}\widehat C(C_+C_-)^2\|\hat{b}^{1/2}\hat{v}_e\|_{0,\hat e}.
\end{align*}
As in the element case, with $C_W:=n_+^3\widehat{C}_\lambda C_2^{2-n_{+}} C_{-} + n_-^3C_{+}\widehat{C}_\lambda C_2^{2-n_{-}}$, Lemma \ref{lem:What} yields
\begin{equation}\label{eq:Theta_grad_bound}
    \begin{aligned}
    \|(\widehat{W}^\pm)^\prime\|_{L^\infty(\hat e)}\le& n_\pm^3 \widehat{C}_\lambda C_2^{2-n_{\pm}}h_e^{\,n_\pm-2},\\
    \|(\widehat W^+\widehat W^-)^\prime\|_{L^\infty(\hat e)}\le& C_Wh_e^{\,n_{+}+n_{-}-4}.
    \end{aligned}
\end{equation}
It follows from $((\widehat W^+\widehat W^-)^{-2})^\prime=-2(\widehat{W}^+\widehat{W}^-)^{-3}(\widehat{W}^+\widehat{W}^-)^\prime$, and \eqref{eq:edge_W_bounds}, \eqref{eq:phat}, \eqref{eq:Theta_grad_bound} that
\begin{align*}
    \|\hat{v}^\prime\|_{0,\hat e} 
    =& \| (\widehat{W}^+\widehat{W}^-)^{-2}\hat{p}^\prime_e
     + ((\widehat W^+\widehat W^-)^{-2})^\prime\hat{p}_e\|_{0, \hat e}\\
    \leq&
    \bigl((c_{+}c_{-})^{-2} +2(c_{+}c_{-})^{-3}C_W\bigr)\widehat{C} h_e^{-2n_+-2n_-+8}\|\hat{b}^{1/2}\hat{p}_e\|_{0,\hat e}\\
    \leq& \bigl((c_{+}c_{-})^{-2}  +2(c_{+}c_{-})^{-3}C_W\bigr)\widehat{C} C_{+}^2C_{-}^2\|\hat{b}^{1/2}\hat{v}\|_{0,\hat e}.
\end{align*}
The proof is complete. \qed
\end{proof}

Inverse inequalities on a physical edge $e$ follow immediately.

\begin{theorem}
Let $\mathbf t_e$ be the unit tangent to $e$. For any $v\in V_e$, it holds that
\begin{subequations}
    \begin{align}
    \left\|\partial_{\mathbf t_e}v\right\|_{0,e}&\lesssim h_e^{-1}\|v\|_{0,e},\label{eq:vprime_edgeinverse}\\
\|v\|_{L^\infty(e)}&\lesssim h_e^{-1/2}\|v\|_{0,e}.\label{eq:v_edgeinverse}
\end{align}
\end{subequations}
\end{theorem}

\begin{proof}
By Lemma \ref{lem:ref_edge_inverse} and $|F_e^\prime|\eqsim h_e$, we have
\[
\|\partial_{\mathbf{t}_e}v\|_{0,e}
\eqsim h_e^{-1/2}\|\hat{v}^\prime\|_{0,\hat e}
\lesssim h_e^{-1/2}\|\hat v\|_{0,\hat e}
\eqsim h_e^{-1}\|v\|_{0,e}.
\]
For the same reason, we obtain
\[
\|v\|_{L^\infty(e)}
\eqsim \|\hat v\|_{L^\infty(\hat e)}
\lesssim \|\hat v\|_{0,\hat e}
\eqsim h_e^{-1/2}\|v\|_{0,e}.
\]
The proof is complete. \qed
\end{proof}

Next, we derive a lower bound for the residual estimator. As in the triangular case, the proof is based on bubble functions. For $T\in\mathcal{T}_h$, let
\[
b_T=
\begin{cases}
27\lambda_{1,T}\lambda_{2,T}\lambda_{3,T},& \text{in }T,\\
0,& \text{in }\Omega\setminus T.
\end{cases}
\]
Let $e=\partial T^+\cap \partial T^-$ be an interior edge with endpoints 
$\mathbf a$ and $\mathbf b$. We consider the edge bubble function
\[
b_e=4\lambda_{\mathbf a}\lambda_{\mathbf b}.
\]
Then $0\le b_T\leq 1$, $b_T|_{\partial T}=0$ and $0\leq b_e\leq 1$, ${\rm supp}(b_e)=\Omega_e:=T^+\cup T^-$. 
\begin{lemma}\label{lem_T_bubble}
For $T\in\mathcal T_h$, $e\in\mathcal E_h$ and $v\in V_T$, $w\in V_e$, it holds that
\begin{subequations}
    \begin{align}
        &\|v\|_{0,T}\lesssim \|b_T^{1/2}v\|_{0,T}\leq \|v\|_{0,T},\label{eq_elem_bubble}\\
        &\|w\|_{0,e}\lesssim \|b_e^{1/2}w\|_{0,e}\le \|w\|_{0,e},\label{eq_edge_bubble}\\
&|b_Tv|_{1,T}\lesssim h_T^{-1}\|v\|_{0,T}.\label{eq_elem_bubble_grad}
    \end{align}
\end{subequations}
\end{lemma}

\begin{proof}
The second inequality in both \eqref{eq_elem_bubble} and \eqref{eq_edge_bubble} follow immediately from $0\le b_T\leq 1$ and $0\le b_e\leq 1$. 
Using Lemmas \ref{lem:Jacobian_det}, \ref{lem:full_inverse_reference} and \ref{lem:What}, we obtain
\begin{align*}
    \|v\|_{0,T}\eqsim& h_T\|\hat v\|_{0,\widehat T}
    \lesssim h_T\|(\lambda_{1,\hat{T}}\lambda_{2,\hat{T}}\lambda_{3,\hat{T}})^{(n-2)/2}\hat v\|_{0,\widehat T}\\
    \lesssim& h_T\|(\hat w_1\hat w_2\hat w_3/\widehat W^3)^{1/2}\hat v\|_{0,\widehat T}
    = h_T\|\hat b_T^{1/2}\hat v\|_{0,\widehat T}
    \eqsim \|b_T^{1/2}v\|_{0,T}.
\end{align*}
Similarly, by Lemma \ref{lem:ref_edge_inverse}, we obtain
\[
\|w\|_{0,e}\eqsim h_e^{1/2}\|\hat w\|_{0,\hat e}
\lesssim h_e^{1/2}\|\hat b_e^{1/2}\hat w\|_{0,\hat e}
\eqsim \|b_e^{1/2}w\|_{0,e}.
\]
Finally, \eqref{eq_elem_bubble_grad} follows from Lemmas \ref{lem:wach_grad_bound} and \ref{thm:inverse}. \qed
\end{proof}

We now prove the efficiency of the residual estimator stated in Theorem~\ref{thm:efficiency}. 
\begin{proof} 
By Lemma \ref{lem_T_bubble}, integration by parts and $f=-\Delta u$, we have
\begin{equation}\label{eq:data_oscillation}
\begin{aligned}
&\|f_T+\Delta u_h\|_{0,T}^2\lesssim \|b_T^{1/2}(f_T+\Delta u_h)\|_{0,T}^2\\
&=
-\int_T \Delta(u-u_h)(f_T+\Delta u_h)b_T-\int_T (f-f_T)(f_T+\Delta u_h)b_T\\
&=
-\int_T \nabla(u-u_h)\cdot \nabla\bigl((f_T+\Delta u_h)b_T\bigr)-\int_T (f-f_T)(f_T+\Delta u_h)b_T.
\end{aligned}
\end{equation}
Using Lemma \ref{lem_T_bubble}, we have
\begin{equation*}
|(f_T+\Delta u_h)b_T|_{1,T}\lesssim
h_T^{-1}\|f_T+\Delta u_h\|_{0,T}.
\end{equation*}
Substituting this into \eqref{eq:data_oscillation} yields
\begin{equation*}
\|f_T+\Delta u_h\|_{0,T}
\lesssim
h_T^{-1}|u-u_h|_{1,T}+\|f-f_T\|_{0,T}.
\end{equation*}
Employing a triangle inequality then yields
\begin{equation}\label{eq:res_volume}
\|f+\Delta u_h\|_{0,T}
\lesssim
h_T^{-1}|u-u_h|_{1,T}+\|f-f_T\|_{0,T}.
\end{equation}

For $e=[\mathbf{a},\mathbf{b}]\in\mathcal{E}_h$, we need to estimate $\|\llbracket\partial_{\mathbf{n}_e} u_h\rrbracket\|_{0,e}$. Along the direction of $e$,
let $q|_e:=\llbracket\partial_{\mathbf{n}_e} u_h\rrbracket$ and $q|_{e^c}$ outside $e$ be given by the endpoint values $q(\mathbf{a})$, $q(\mathbf{b})$. We then constantly extend $q$ into $\Omega_e$ along the normal vector $\mathbf{n}_e$ to obtain $\tilde{q}$. By the construction of $\tilde{q}$ and \eqref{eq:v_edgeinverse},
\begin{equation}\label{eq:qtilde}
    \|\tilde q\|_{0,\Omega_e}^2
\lesssim h_e\|q\|_{0,e}^2
+
h_e^2\|q\|_{L^\infty(e)}^2
\lesssim
h_e\|q\|_{0,e}^2.
\end{equation}
It follows from constant normal extension and \eqref{eq:vprime_edgeinverse} that
\begin{equation}\label{eq:grad_qtilde}
\|\nabla\tilde{q}\|_{0,\Omega_e}^2
= \left\|\partial_{\mathbf{t}_e}\tilde{q}\right\|_{0,\Omega_e}^2
\lesssim h_e\|\partial_{\mathbf{t}_e}q\|_{0,e}^2\lesssim h_e^{-1}\|q\|_{0,e}^2.
\end{equation}

By Lemma \ref{lem_T_bubble}, Lemma \ref{lem:wach_grad_bound} and \eqref{eq:qtilde}, \eqref{eq:grad_qtilde}, we have
\begin{equation*}
\begin{aligned}
\|q\|_{0,e}^2&\lesssim\|b_e^{1/2}q\|_{0,e}^2
=
\sum_{T\in\Omega_e}\int_T \nabla(u-u_h)\cdot \nabla(b_e\tilde q)
-
\sum_{T\in\Omega_e}\int_T (f+\Delta u_h)b_e\tilde{q}\\
&\leq|u-u_h|_{1,\Omega_e}\|\nabla(b_e\tilde{q})\|_{0,\Omega_e}+\|f+\Delta u_h\|_{0,\Omega_e}\|b_e\tilde{q}\|_{0,\Omega_e}\\
&\leq h_e^{-1/2}|u-u_h|_{1,\Omega_e}\|q\|_{0,e}+h_e^{1/2}\|f+\Delta u_h\|_{0,\Omega_e}\|q\|_{0,e}.
\end{aligned}
\end{equation*}
In what follows,
\begin{equation}\label{eq:res_edge}
h_e^{1/2}\|q\|_{0,e}\lesssim |u-u_h|_{1,\Omega_e}+h_e\|f+\Delta u_h\|_{0,\Omega_e}.
\end{equation}
Combining \eqref{eq:res_volume} with \eqref{eq:res_edge} proves Theorem \ref{thm:efficiency}. \qed
\end{proof}

\section{Numerical Experiments}\label{sect:numeric}

In this section, we present three representative examples to assess the proposed adaptive method on polygonal meshes. We first briefly describe the polytree-type local refinement procedure used in the computation, including the treatment of hanging nodes under the constraint that each edge contains at most one hanging node. We then compare adaptive and uniform refinements in terms of mesh distribution, numerical accuracy, and convergence behavior. For each example, $\mathrm{Dof}$ denotes the number of degrees of freedom.

\subsection{A polytree-type mesh refinement strategy}

We employ a polytree-type local refinement strategy for polygonal meshes. For a marked $n$-gon $T$ with vertices $\mathbf a_1,\dots,\mathbf a_n$ ordered counterclockwise, let $\mathbf m_i$ be the midpoint of the edge $[\mathbf a_i,\mathbf a_{i+1}]$. Define the cell center by
\[
\mathbf x_c=\frac1n\sum_{i=1}^n \mathbf a_i,
\]
and let $\mathbf x_i$ be the midpoint of the segment $[\mathbf x_c,\mathbf m_i]$. Connecting $\mathbf x_1,\dots,\mathbf x_n$ produces an interior $n$-gon, and connecting each $\mathbf x_i$ to $\mathbf m_i$ subdivides $T$ into $n$ outer pentagons and one inner $n$-gon. Hence one refinement of an $n$-gon produces $n+1$ polygonal sub-elements.

Hanging nodes are unavoidable under local mesh refinement. In our implementation, hanging nodes are treated only as auxiliary topological points and are not taken as independent degrees of freedom. Their nodal values are determined by interpolation from the two endpoints of the containing edge. This avoids nonconvex coarse cells and preserves the validity of Wachspress coordinates. Moreover, we enforce that each edge contains at most one hanging node; if an endpoint of an edge is already hanging, the adjacent coarse element is refined accordingly. Our a posteriori error analysis with minor modification extends to such nonconforming polygonal meshes. 

The elements to be refined are selected by the D\"orfler marking strategy. More precisely, for a given marking parameter $\theta\in(0,1)$, we choose a minimal subset $\mathcal M_h\subset \mathcal T_h$ such that
\[
\sum_{T\in\mathcal M_h}\eta_T^2
\geq
\theta \sum_{T\in\mathcal T_h}\eta_T^2 .
\]
The elements in $\mathcal M_h$ are then refined by the above polytree-type subdivision. In all adaptive computations reported below, we take $\theta=0.6$.

\subsection{Example 1: a localized peak on the unit square}

Let $\Omega=(0,1)^2$ and take
\[
u(x,y)=10e^{-\beta((x-x_0)^2+(y-y_0)^2)},
\qquad (x_0,y_0)=(0.5,0.5),
\]
with $\beta=10$. The right-hand side is given by $-\Delta u=f$, and the boundary data are prescribed by the exact solution.

\begin{figure}[htbp]
  \centering
  \begin{minipage}{0.32\textwidth}
    \centering
    \includegraphics[width=\linewidth]{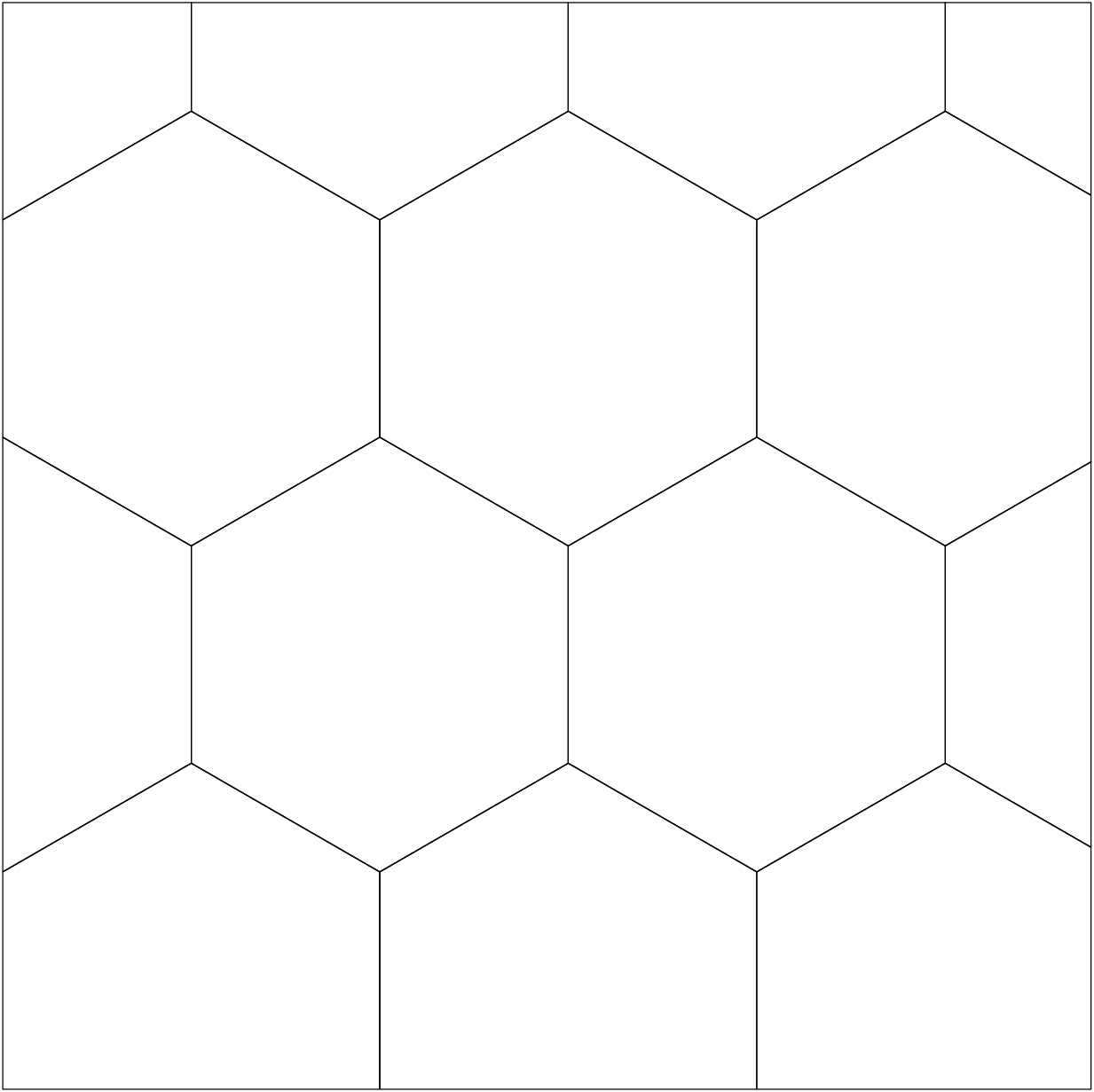}
    
    \small (a) Initial mesh.
  \end{minipage}
  \hfill
  \begin{minipage}{0.32\textwidth}
    \centering
    \includegraphics[width=\linewidth]{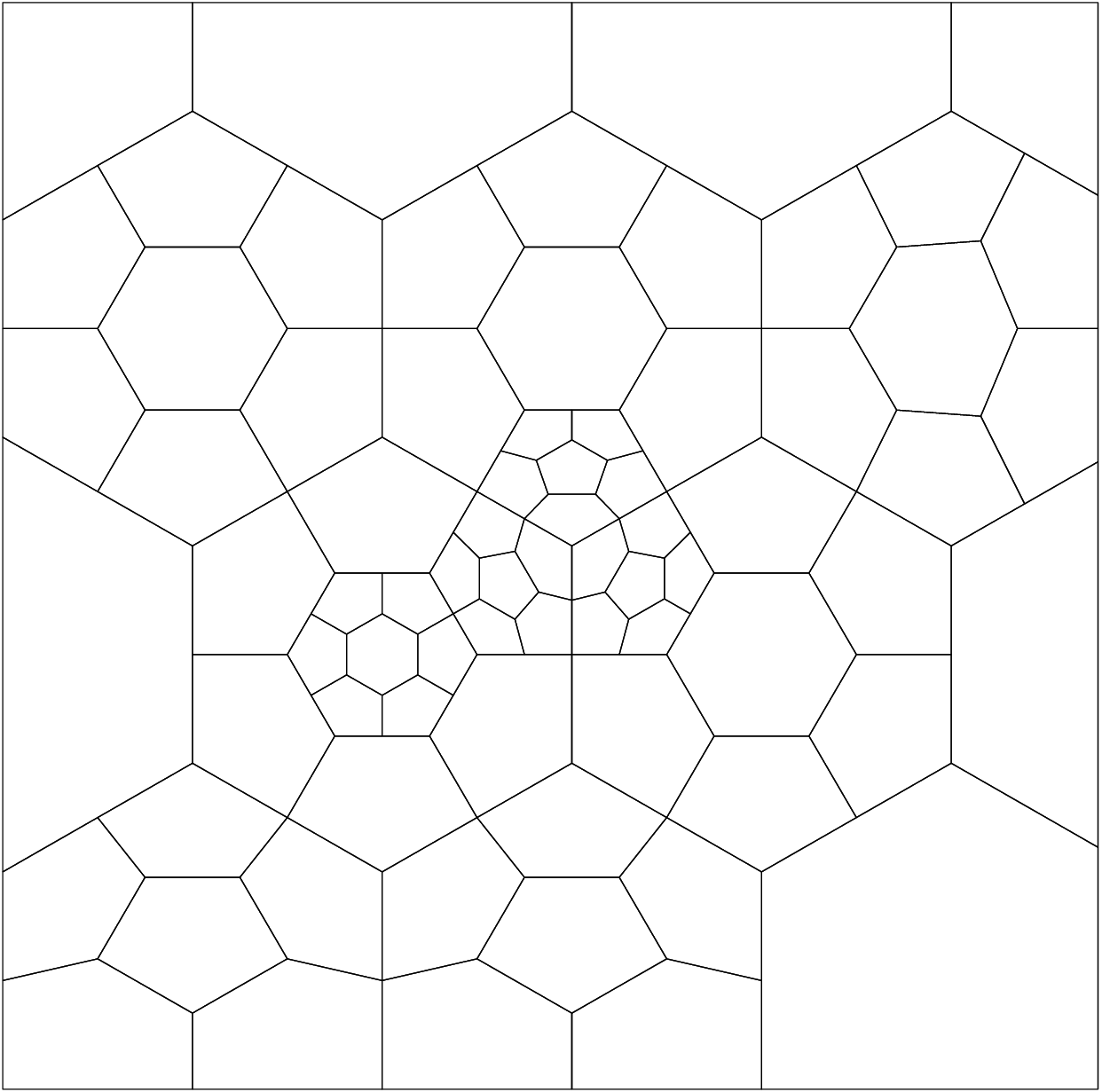}
    
    \small (b) Mesh after $4$ iterations.
  \end{minipage}
  \hfill
  \begin{minipage}{0.32\textwidth}
    \centering
    \includegraphics[width=\linewidth]{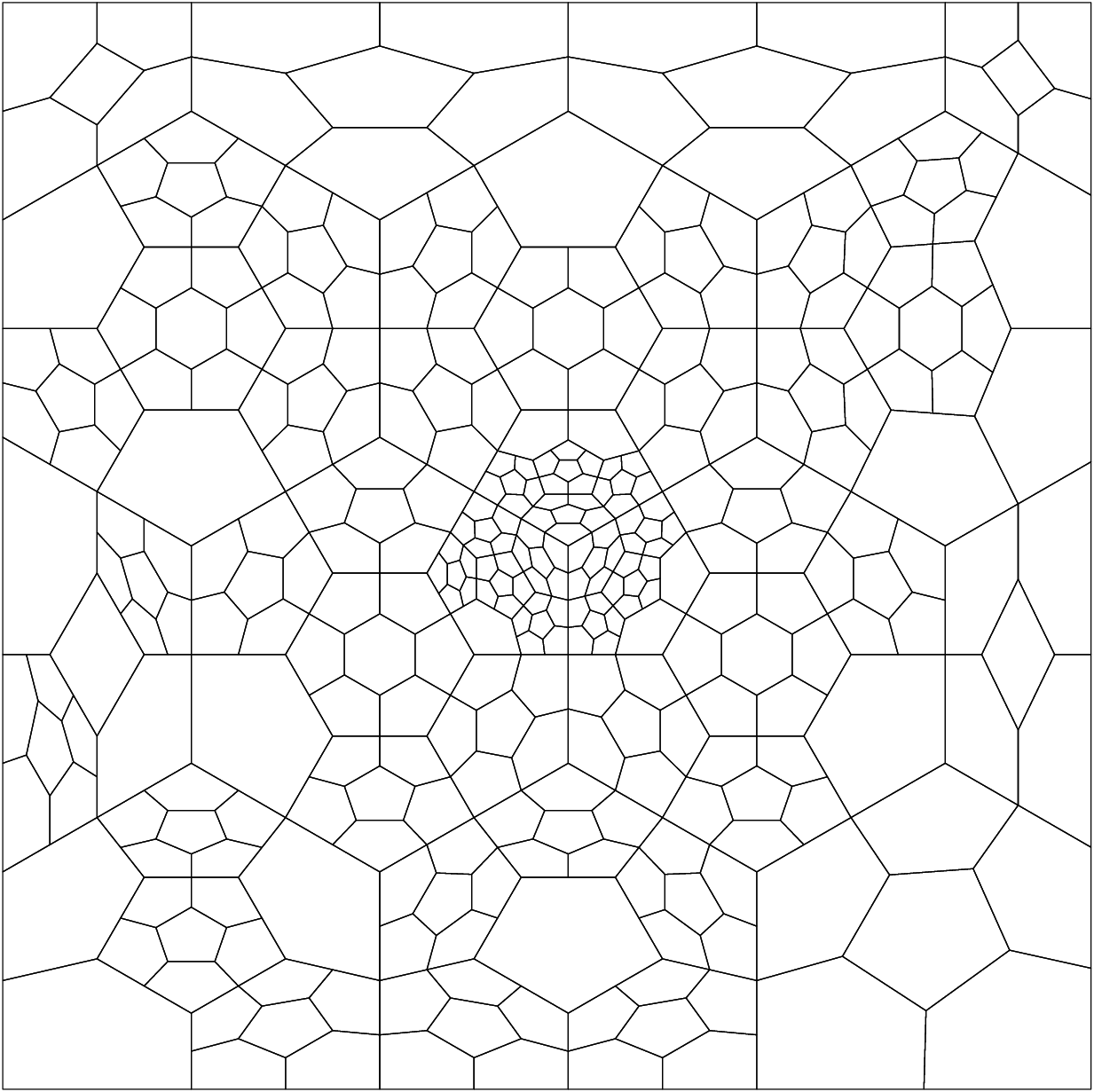}
    
    \small (c) Mesh after $8$ iterations.
  \end{minipage}
  \caption{Adaptive mesh refinement for Example 1.}
  \label{fig:ex1_mesh}
\end{figure}

\begin{figure}[htbp]
  \centering
  \includegraphics[width=7cm]{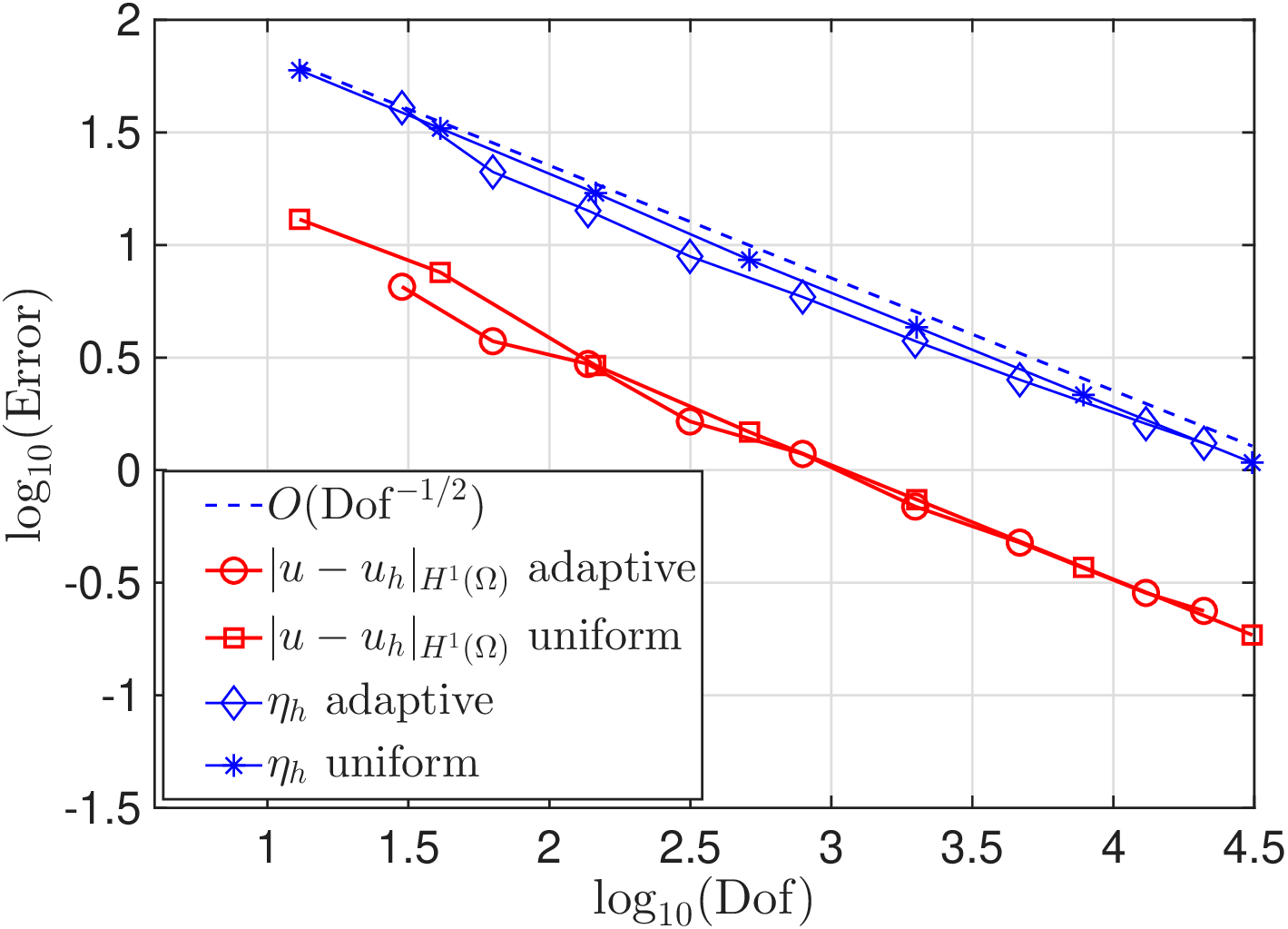}
  \caption{Convergence history in Example 1.}
  \label{fig:ex1_conv}
\end{figure}

\begin{table}[htbp]
  \centering
  \caption{Errors and estimators for uniform and adaptive refinements in Example 1}
  \label{tab:ex1_compare}
  \begin{tabular}{ccccccc}
    \toprule
    \multicolumn{3}{c}{Uniform mesh} & \multicolumn{1}{c}{} & \multicolumn{3}{c}{Adaptive mesh} \\
    \cmidrule(lr){1-3}\cmidrule(lr){5-7}
    $\mathrm{Dof}$ & $|u-u_h|_{1,\Omega}$ & $\eta_h$
    & &
    $\mathrm{Dof}$ & $|u-u_h|_{1,\Omega}$ & $\eta_h$ \\
    \midrule
        41    & 7.5609 & 32.9851 &&    53    & 4.7584 & 28.2927 \\
       146    & 2.9145 & 17.0061 &&   137    & 2.9623 & 14.2203 \\
       512    & 1.4766 &  8.5969 &&   568    & 1.4183 &  7.0823 \\
      2008    & 0.7399 &  4.3122 &&  1983    & 0.6890 &  3.7446 \\
      7840    & 0.3705 &  2.1593 &&  8087    & 0.3361 &  1.9046 \\
     31090    & 0.1853 &  1.0802 && 20957    & 0.2371 &  1.3168 \\
    \bottomrule
  \end{tabular}
\end{table}

The solution has a pronounced local peak near the center of the domain. As shown in Figure~\ref{fig:ex1_mesh}, the adaptive meshes refine progressively in this region and remain relatively coarse elsewhere. The convergence curves in Figure~\ref{fig:ex1_conv} indicate that both the $H^1$-error and the estimator decay essentially with the reference rate $\mathrm{Dof}^{-1/2}$. Table~\ref{tab:ex1_compare} further shows that adaptive refinement yields smaller errors and estimators than uniform refinement for comparable degrees of freedom.

\subsection{Example 2: an interior-layer problem on the unit square}

Let $\Omega=(0,1)^2$ and choose
\[
u(x,y)
=
x(1-x)y(1-y)\arctan(60(r-1)),
\quad
r^2=(x-1.25)^2+(y+0.25)^2 .
\]
The right-hand side is determined by $-\Delta u=f$, and homogeneous Dirichlet boundary data are prescribed on $\partial\Omega$.

\begin{figure}[htbp]
  \centering
  \begin{minipage}{0.32\textwidth}
    \centering
    \includegraphics[width=\linewidth]{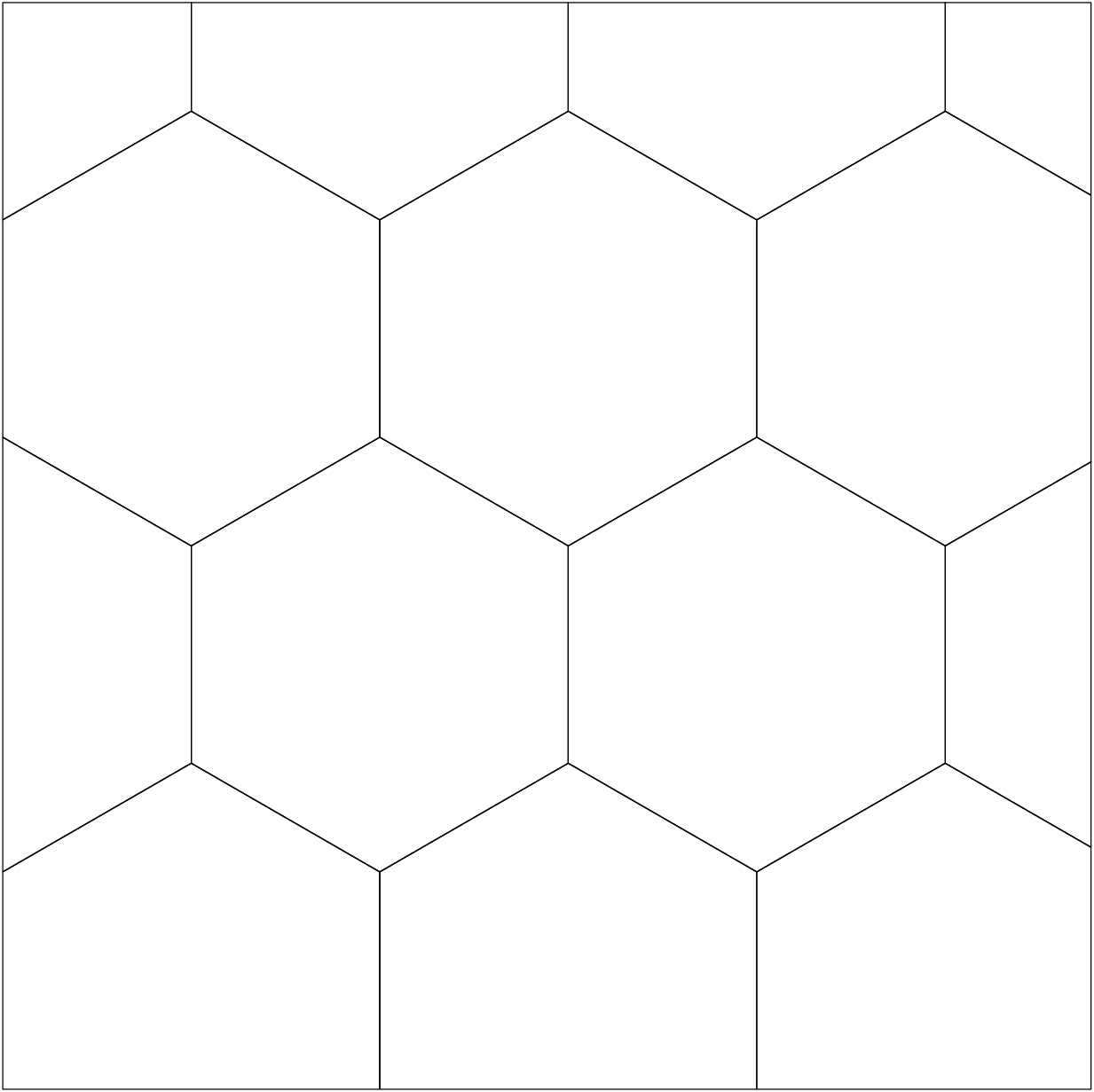}

    \small (a) Initial mesh. 
  \end{minipage}
  \hfill
  \begin{minipage}{0.32\textwidth}
    \centering
    \includegraphics[width=\linewidth]{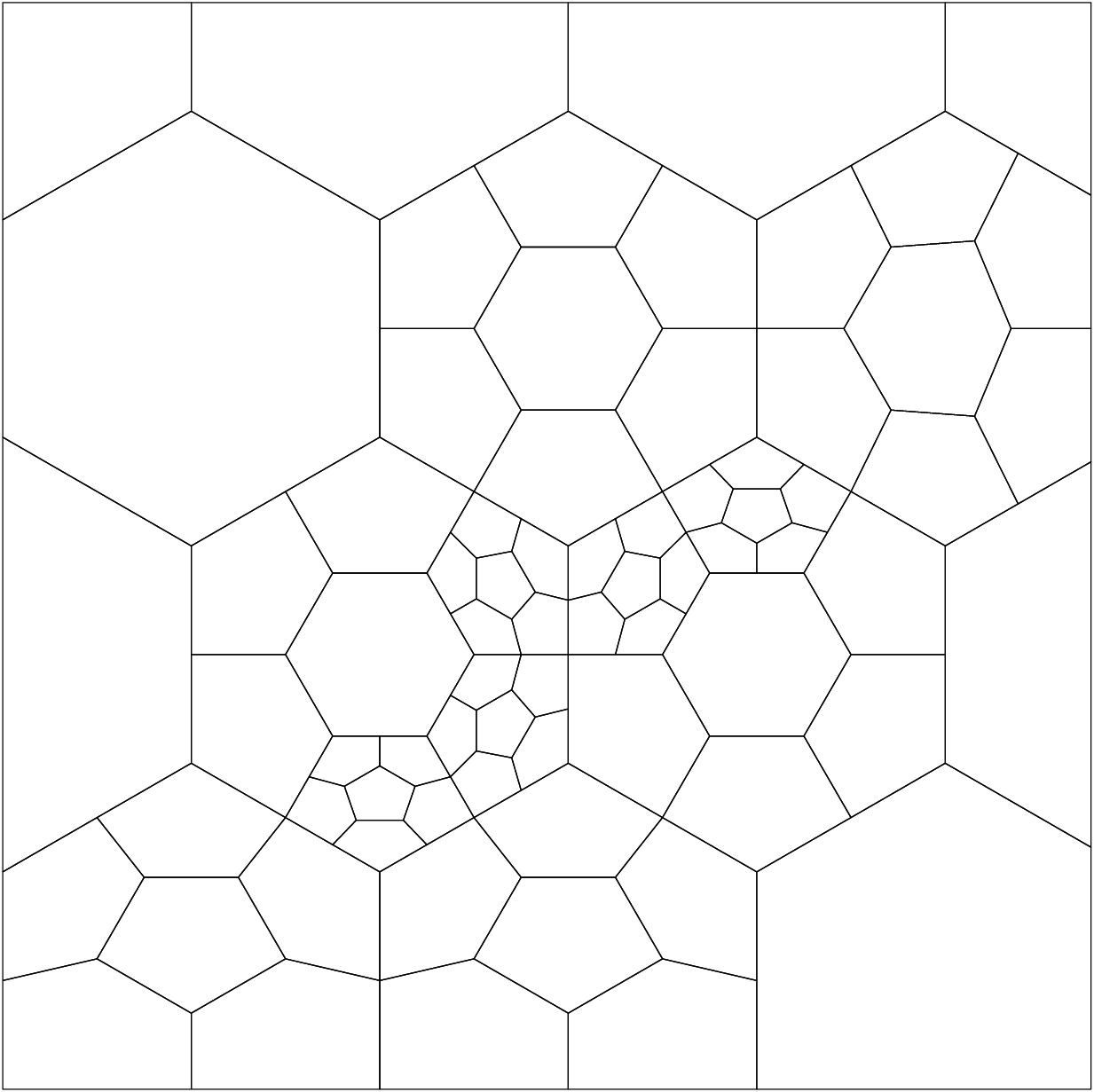}

    \small (b) Mesh after $4$ iterations.
  \end{minipage}
  \hfill
  \begin{minipage}{0.32\textwidth}
    \centering
    \includegraphics[width=\linewidth]{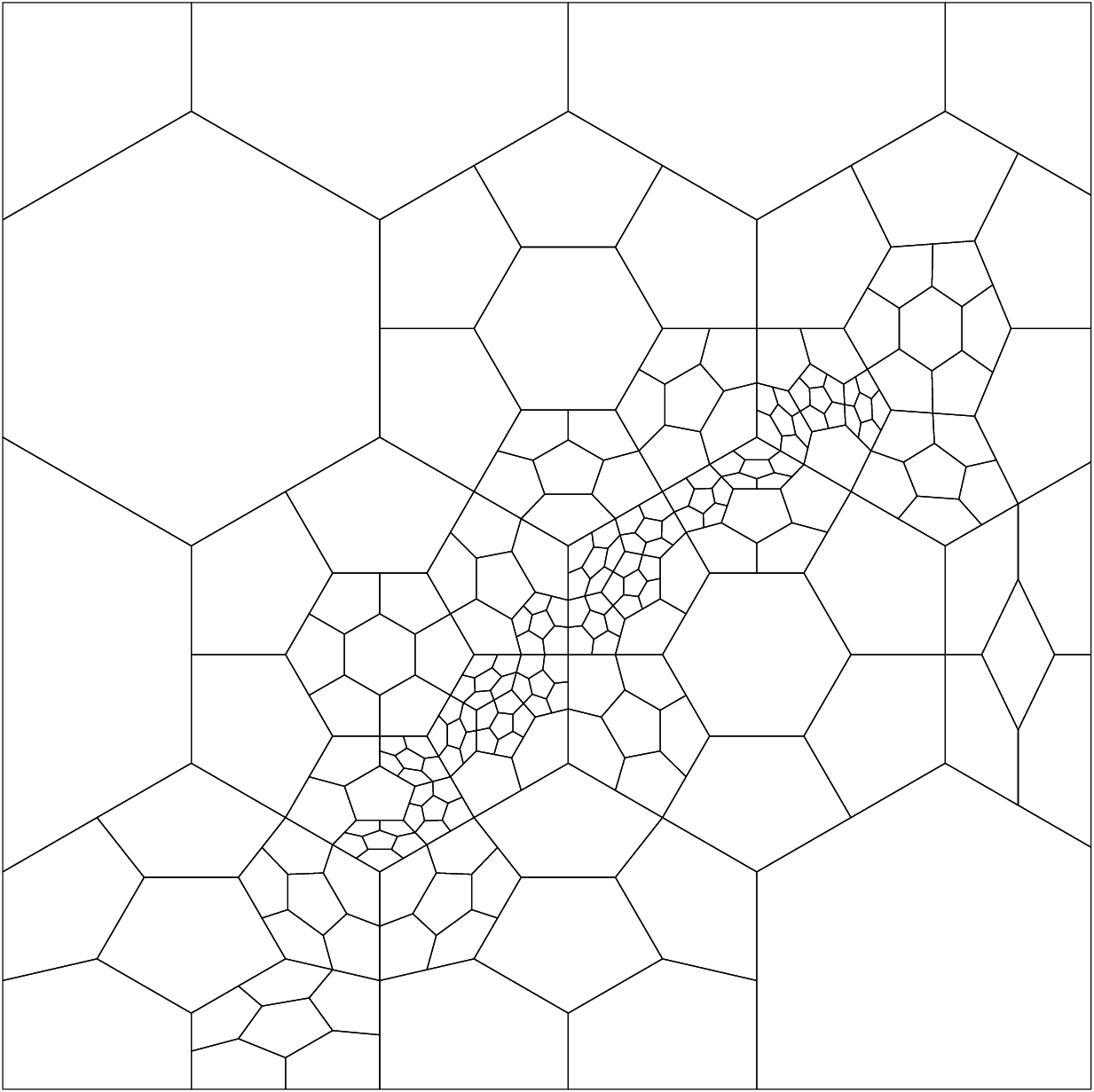}

    \small (c) Mesh after $8$ iterations.
  \end{minipage}
  \caption{Adaptive mesh refinement in Example 2.}
  \label{fig:ex2_mesh}
\end{figure}

\begin{figure}[htbp]
  \centering
  \begin{minipage}{0.48\textwidth}
    \centering
    \includegraphics[width=\linewidth]{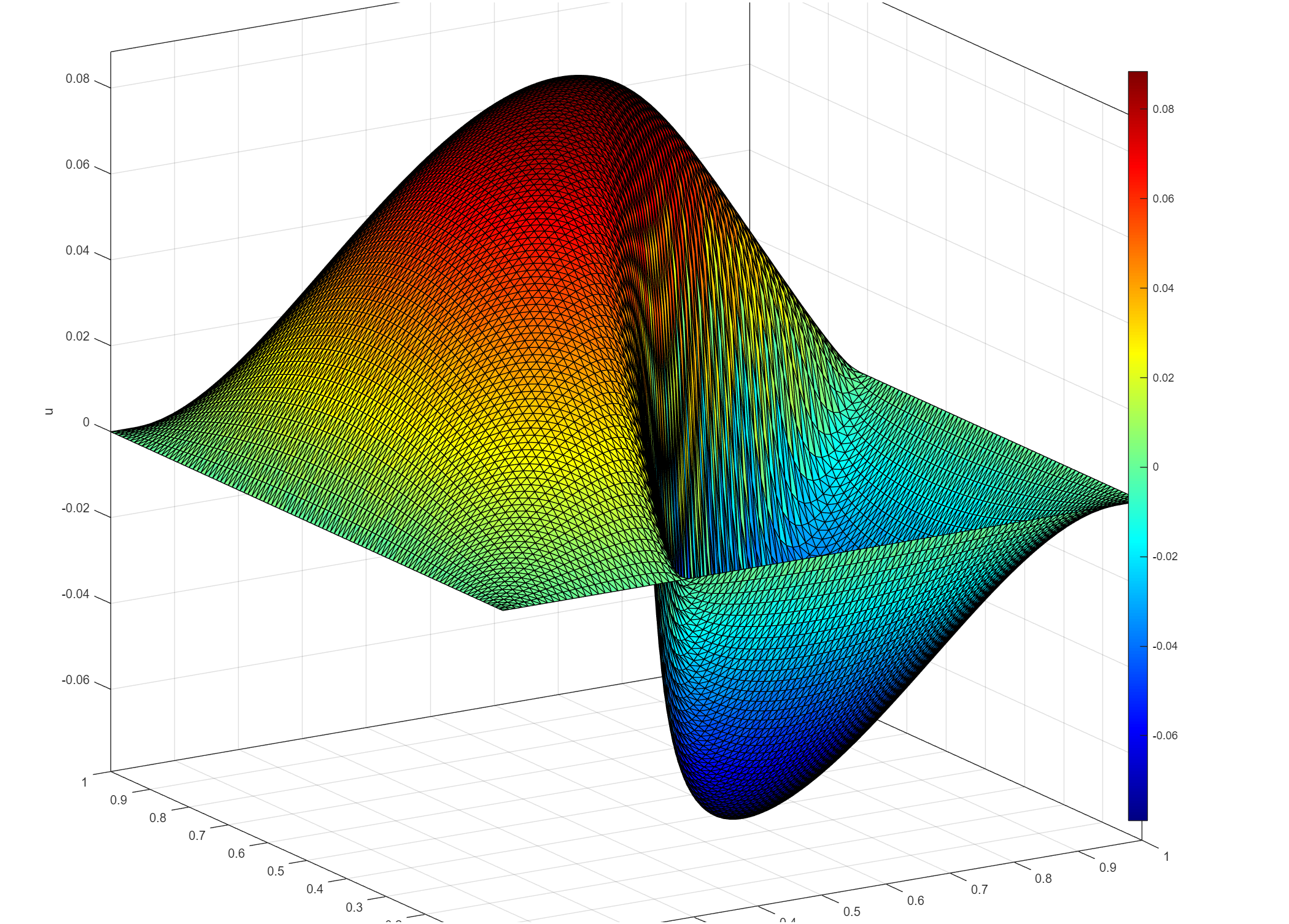}
    
    \vspace{0.15cm}
    \small (a) Exact solution
  \end{minipage}
  \hfill
  \begin{minipage}{0.38\textwidth}
    \centering
    \includegraphics[width=\linewidth]{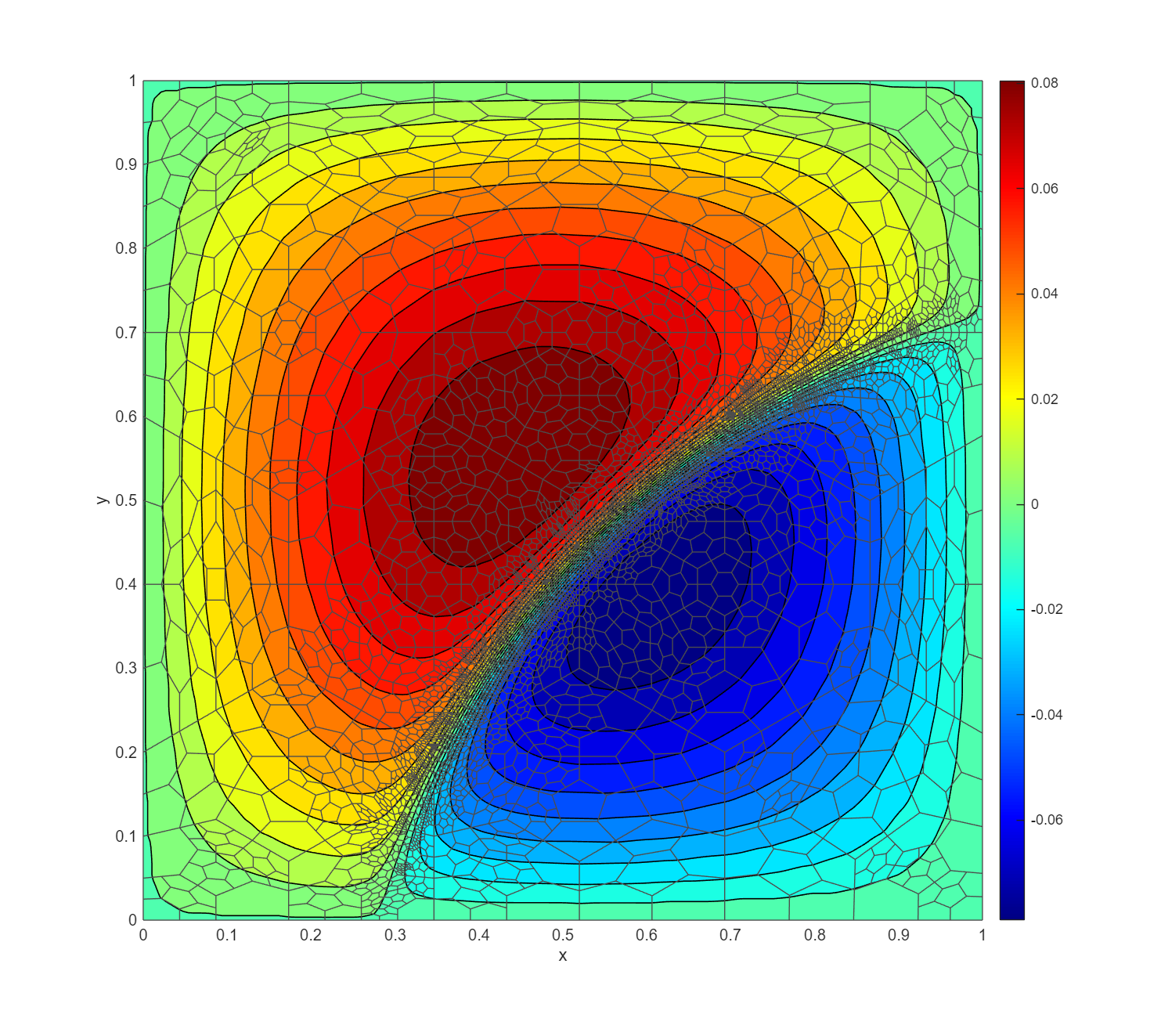}
    
    \vspace{0.15cm}
    \small (b) Numerical solution
  \end{minipage}
  \caption{Exact and numerical solutions in Example 2.}
  \label{fig:ex2_sol}
\end{figure}

\begin{figure}[htbp]
  \centering
  \includegraphics[width=7cm]{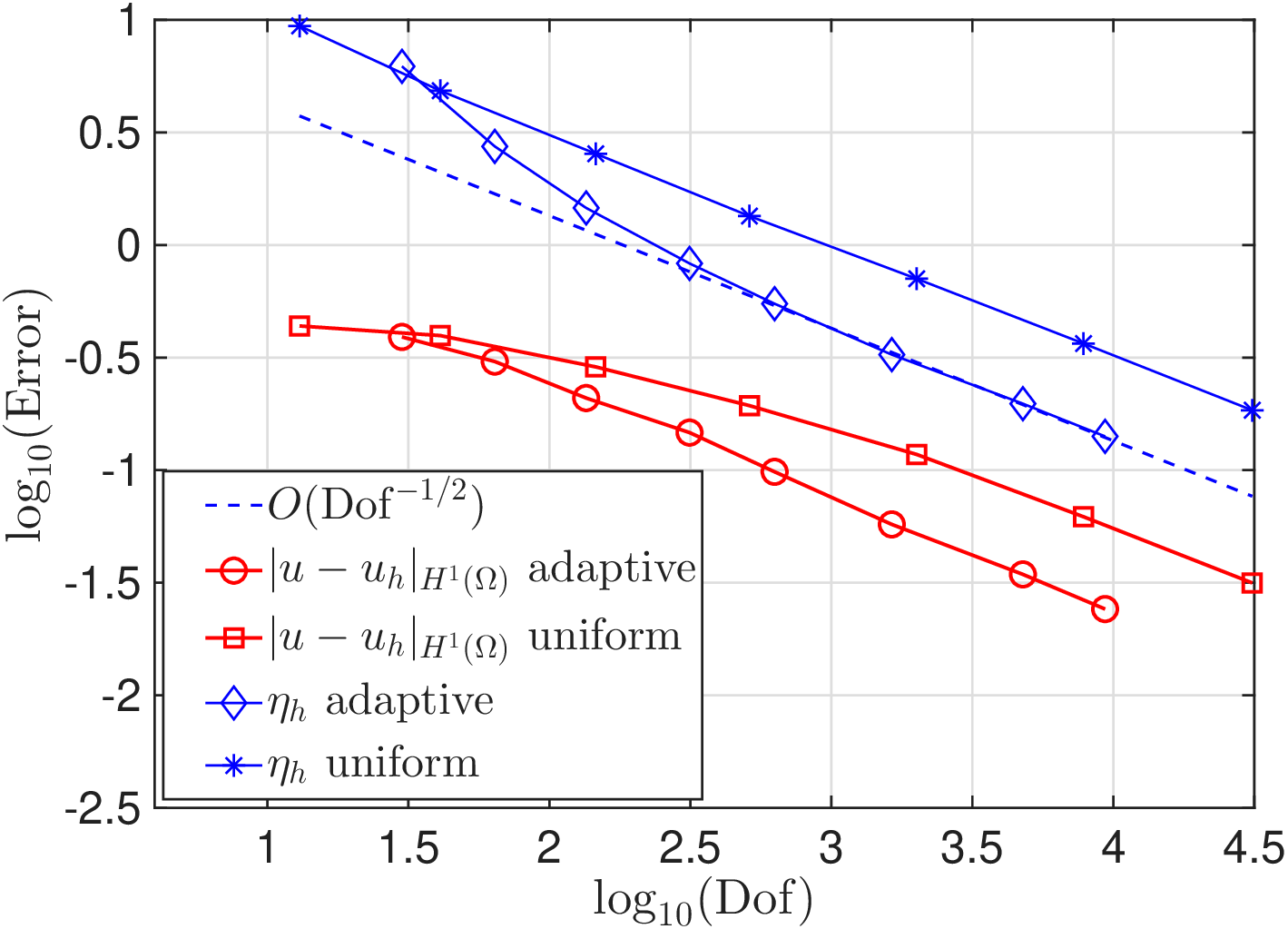}
  \caption{Convergence history in Example 2.}
  \label{fig:ex2_conv}
\end{figure}

\begin{table}[htbp]
  \centering
  \caption{Convergence of errors and error estimators in Example 2.}
  \label{tab:ex2_compare}
  \begin{tabular}{ccccccc}
    \toprule
    \multicolumn{3}{c}{Uniform mesh} & \multicolumn{1}{c}{} & \multicolumn{3}{c}{Adaptive mesh} \\
    \cmidrule(lr){1-3}\cmidrule(lr){5-7}
    $\mathrm{Dof}$ & $|u-u_h|_{1,\Omega}$ & $\eta_h$
    & &
    $\mathrm{Dof}$ & $|u-u_h|_{1,\Omega}$ & $\eta_h$ \\
    \midrule
        41    & 0.6548 & 3.3677 &&    49   & 0.4601 & 2.1687 \\
       146    & 0.3565 & 1.9137 &&   130   & 0.2193 & 1.1093 \\
       512    & 0.1881 & 1.0471 &&   460   & 0.0976 & 0.5499 \\
      2008    & 0.0968 & 0.5526 &&  1958   & 0.0471 & 0.2743 \\
      7840    & 0.0491 & 0.2843 &&  6148   & 0.0280 & 0.1651 \\
     31090    & 0.0248 & 0.1445 &&  8922   & 0.0238 & 0.1429 \\
    \bottomrule
  \end{tabular}
\end{table}

This example exhibits a sharp internal layer near the circular arc. Figure~\ref{fig:ex2_mesh} shows that the adaptive meshes are concentrated in the region where the solution varies rapidly, while the smoother region remains coarse. The numerical solution in Figure~\ref{fig:ex2_sol} agrees well with the exact one. The convergence curves in Figure~\ref{fig:ex2_conv} again show decay close to $\mathrm{Dof}^{-1/2}$, and Table~\ref{tab:ex2_compare} confirms that adaptive refinement achieves higher accuracy than uniform refinement at comparable cost.

\subsection{Example 3: a corner singularity on the L-shaped domain}

Consider the L-shaped domain $\Omega=(-1,1)^2\setminus([0,1)\times(-1,0])$, and the exact solution
\[
u(r,\theta)=r^{2/3}\sin\Bigl(\frac{2\theta}{3}\Bigr),
\]
where $(r,\theta)$ are polar coordinates centered at the origin. Since $u$ is harmonic away from the re-entrant corner, one has $f\equiv0$, and the boundary data are given by the exact solution.

\begin{figure}[htbp]
  \centering
  \begin{minipage}{0.32\textwidth}
    \centering
    \includegraphics[width=\linewidth]{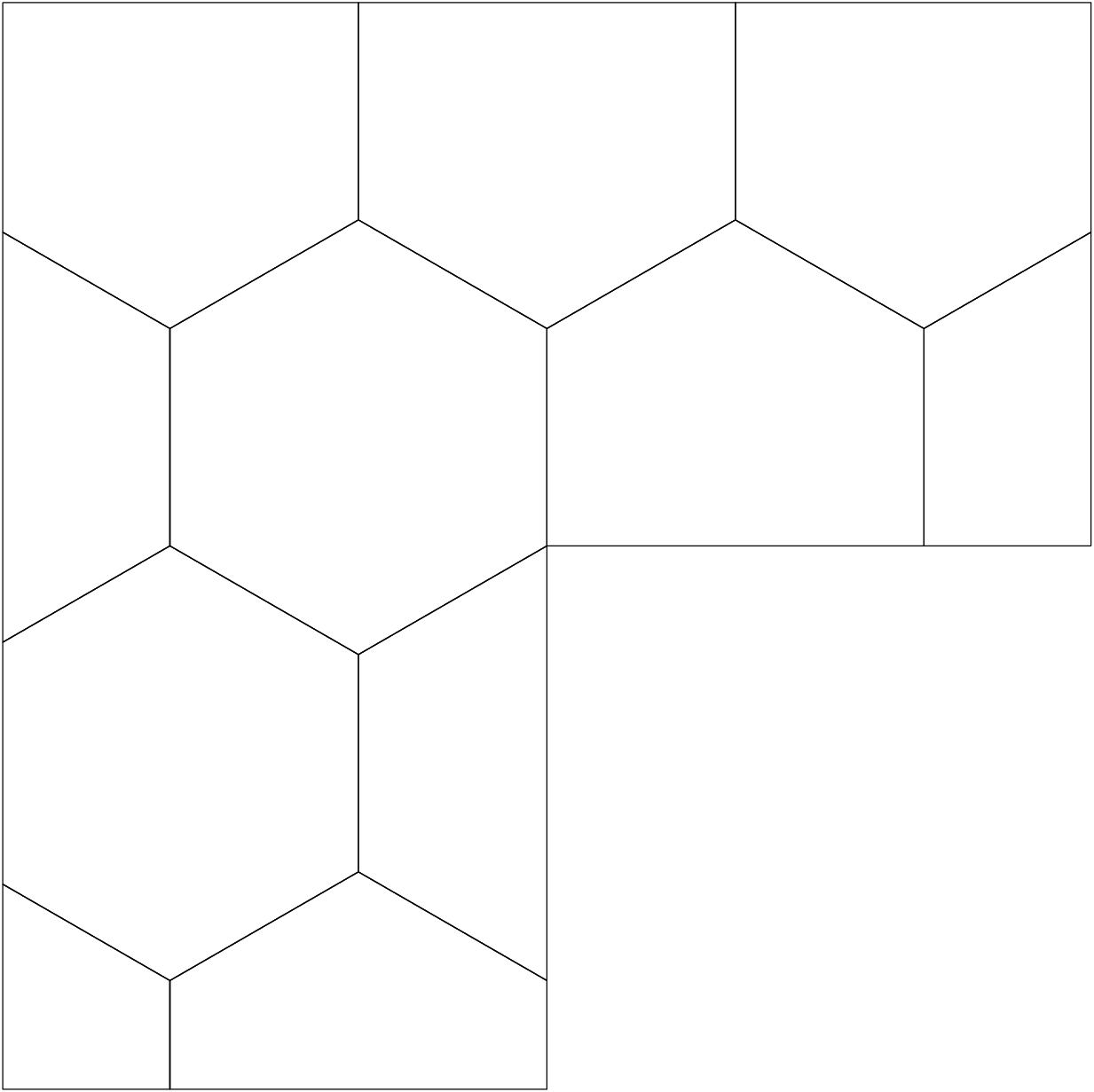}

    \small (a) Initial mesh.
  \end{minipage}
  \hfill
  \begin{minipage}{0.32\textwidth}
    \centering
    \includegraphics[width=\linewidth]{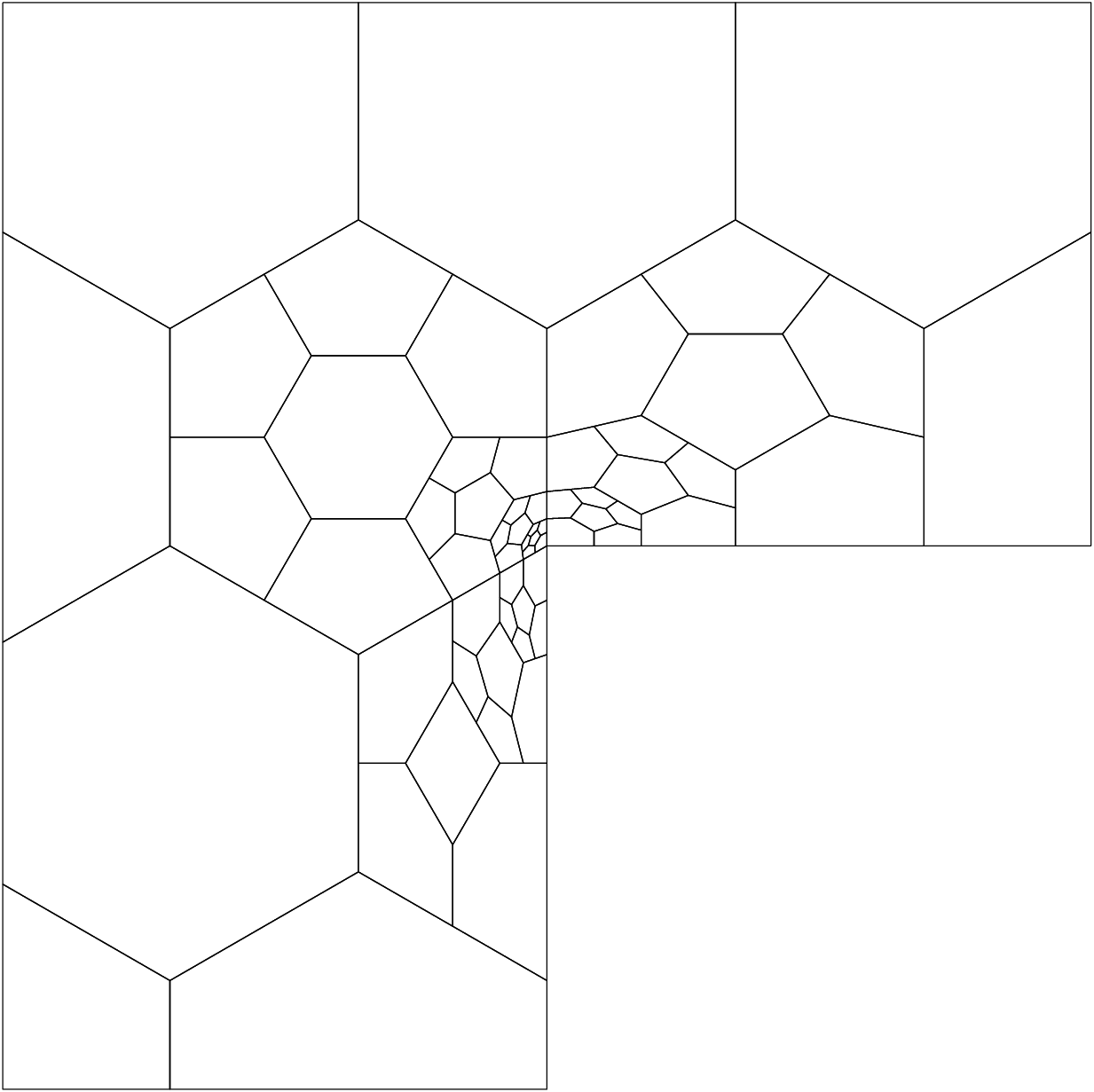}

    \small (b) Mesh after $4$ iterations.
  \end{minipage}
  \hfill
  \begin{minipage}{0.32\textwidth}
    \centering
    \includegraphics[width=\linewidth]{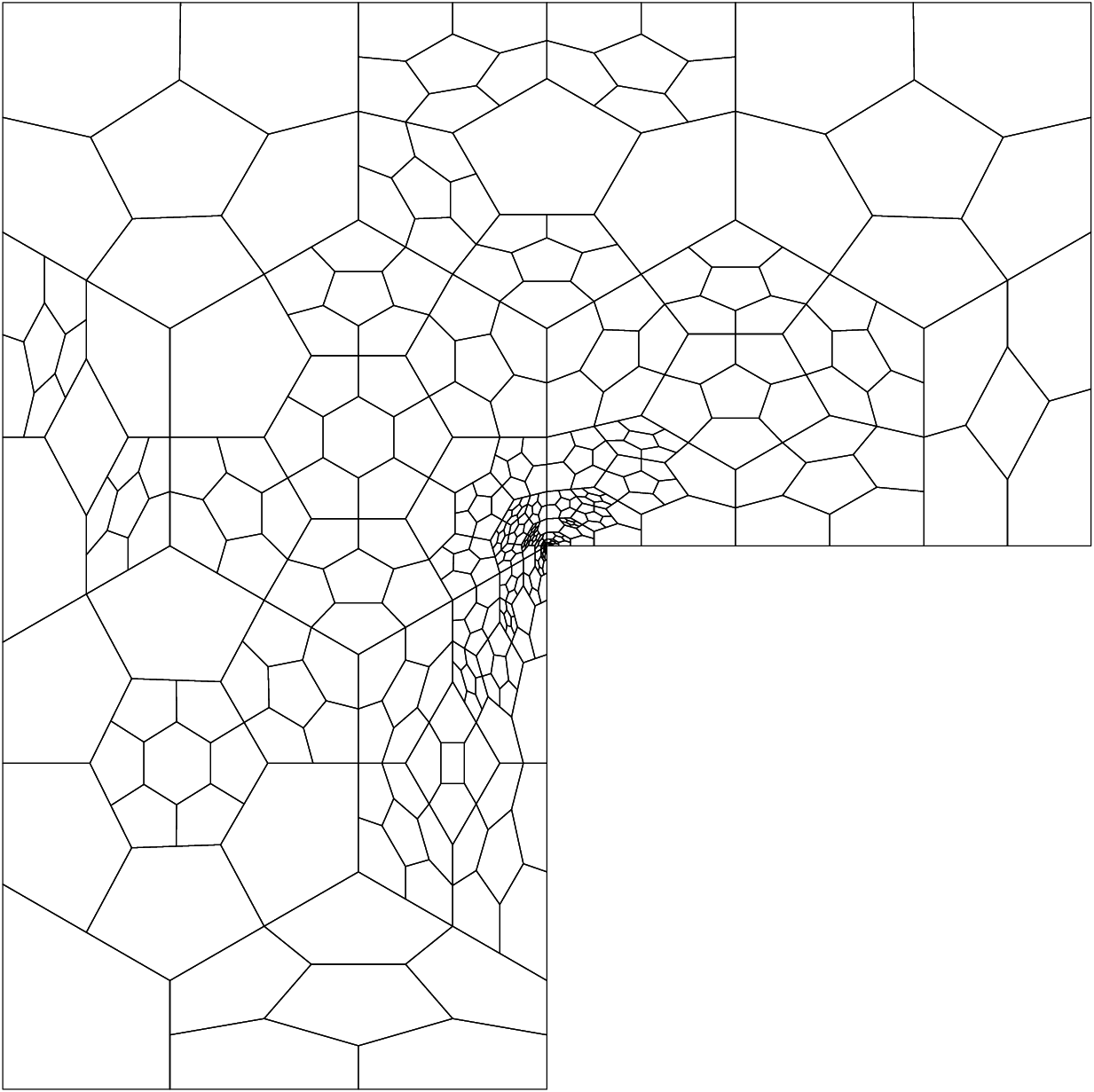}

    \small (c) Mesh after $8$ iterations.
  \end{minipage}
  \caption{Adaptive mesh refinement for Example 3}
  \label{fig:ex4_mesh}
\end{figure}

\begin{figure}[htbp]
  \centering
  \includegraphics[width=7cm]{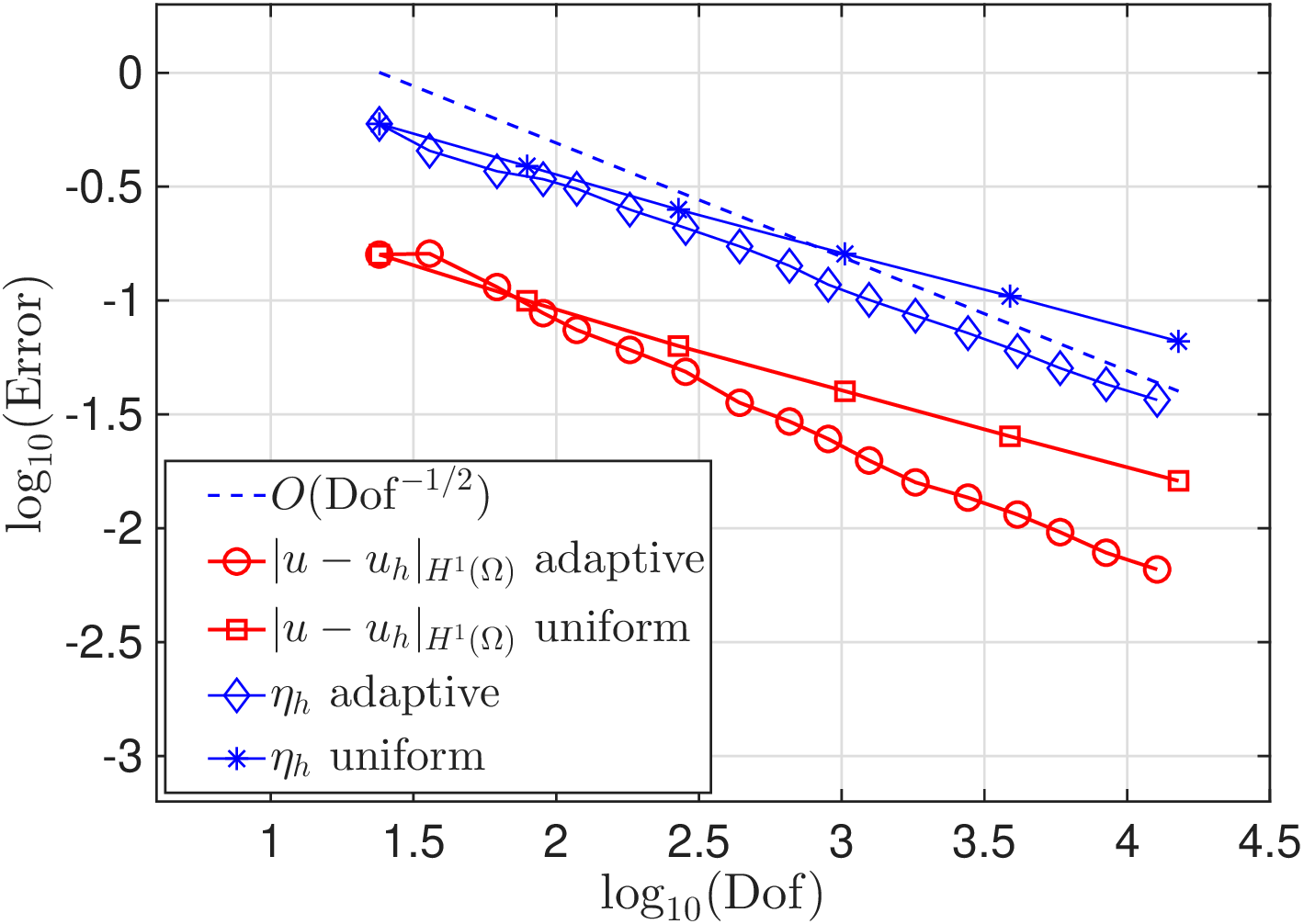}

  \caption{Convergence history in Example 3.}
  \label{fig:ex4_conv}
\end{figure}

\begin{table}[htbp]
  \centering
  \caption{Errors, estimators, and order of convergence in Example 3.}
  \vspace{0.2cm}
  \label{tab:ex4_compare}
  \begin{tabular}{cccc c cccc}
    \toprule
    \multicolumn{4}{c}{Uniform mesh} & &
    \multicolumn{4}{c}{Adaptive mesh} \\
    \cmidrule(lr){1-4}\cmidrule(lr){6-9}
    $\mathrm{Dof}$ & $|u-u_h|_{1,\Omega}$ & order & $\eta_h$
    & &
    $\mathrm{Dof}$ & $|u-u_h|_{1,\Omega}$ & order & $\eta_h$ \\
    \midrule
        24   & 0.1591 & --    & 0.5965 &&     24   & 0.1591 & --    & 0.5965 \\
        79   & 0.0998 & 0.391 & 0.3892 &&     90   & 0.0880 & 0.448 & 0.3410 \\
       268   & 0.0632 & 0.375 & 0.2507 &&    284   & 0.0487 & 0.516 & 0.2082 \\
      1026   & 0.0400 & 0.341 & 0.1603 &&    897   & 0.0247 & 0.591 & 0.1174 \\
      3887   & 0.0253 & 0.343 & 0.1045 &&   4127   & 0.0115 & 0.501 & 0.0600 \\
     15131   & 0.0161 & 0.331 & 0.0662 &&  12734   & 0.0066 & 0.493 & 0.0366 \\
    \bottomrule
  \end{tabular}
\end{table}

Because of the re-entrant corner at the origin, the solution does not belong to $H^2(\Omega)$, and uniform refinement suffers from reduced convergence. Figure~\ref{fig:ex4_mesh} shows that the adaptive meshes are strongly concentrated near the singular corner, while remaining coarse away from it.  The convergence curves in Figure~\ref{fig:ex4_conv} indicate that the adaptive method recovers an almost optimal rate close to $\mathrm{Dof}^{-1/2}$, whereas uniform refinement exhibits a slower decay, approximately $\mathrm{Dof}^{-1/3}$. This is also reflected in Table~\ref{tab:ex4_compare}, which confirms the improved efficiency of the adaptive procedure for corner singularities.


\section*{Acknowledgments} 
This work was supported by the National Natural Science Foundation of China under grant 12471346.

\bibliographystyle{plain}

\begin{thebibliography}{10}

\bibitem{AinsworthOden2000}
Mark Ainsworth and J.~Tinsley Oden.
\newblock {\em A posteriori error estimation in finite element analysis}.
\newblock Pure and Applied Mathematics (New York). Wiley-Interscience [John Wiley \& Sons], New York, 2000.

\bibitem{ArnoldBrezziCockburnMarini2001}
Douglas~N. Arnold, Franco Brezzi, Bernardo Cockburn, and L.~Donatella Marini.
\newblock Unified analysis of discontinuous {G}alerkin methods for elliptic problems.
\newblock {\em SIAM J. Numer. Anal.}, 39(5):1749--1779, 2001/02.

\bibitem{Bank1996}
Randolph~E. Bank.
\newblock Hierarchical bases and the finite element method.
\newblock In {\em Acta numerica, 1996}, volume~5 of {\em Acta Numer.}, pages 1--43. Cambridge Univ. Press, Cambridge, 1996.

\bibitem{BankLi2019}
Randolph~E. Bank and Yuwen Li.
\newblock Superconvergent recovery of {R}aviart-{T}homas mixed finite elements on triangular grids.
\newblock {\em J. Sci. Comput.}, 81(3):1882--1905, 2019.

\bibitem{Bank1993}
Randolph~E. Bank and R.~Kent Smith.
\newblock A posteriori error estimators based on hierarchical bases.
\newblock {\em SIAM J. Numer. Anal.}, 30(4):921--935, 1993.

\bibitem{BeckerRannacher2001}
Roland Becker and Rolf Rannacher.
\newblock An optimal control approach to a posteriori error estimation in finite element methods.
\newblock {\em Acta Numer.}, 10:1--102, 2001.

\bibitem{BeiraoBrezziMariniRusso2023}
Lourenco Beir\~ao~da Veiga, Franco Brezzi, L.~Donatella Marini, and Alessandro Russo.
\newblock The virtual element method.
\newblock {\em Acta Numer.}, 32:123--202, 2023.

\bibitem{BeiraoManzini2015}
Lourenco Beir{\~a}o~da Veiga and Gianmarco Manzini.
\newblock Residual a posteriori error estimation for the virtual element method for elliptic problems.
\newblock {\em ESAIM Math. Model. Numer. Anal.}, 49(2):577--599, 2015.

\bibitem{Beirao2019}
Lourenco Beir{\~a}o~da Veiga, Gianmarco Manzini, and Lorenzo Mascotto.
\newblock A posteriori error estimation and adaptivity in hp virtual elements.
\newblock {\em Numer. Math.}, 143(1):139--175, 2019.

\bibitem{BerroneBorio2017}
Stefano Berrone and Andrea Borio.
\newblock A residual {\it a posteriori} error estimate for the {V}irtual {E}lement {M}ethod.
\newblock {\em Math. Models Methods Appl. Sci.}, 27(8):1423--1458, 2017.

\bibitem{BertrandCarstensenGrassleTran2023}
Fleurianne Bertrand, Carsten Carstensen, Benedikt Gr\"a\ss~le, and Ngoc~Tien Tran.
\newblock Stabilization-free {HHO} a posteriori error control.
\newblock {\em Numer. Math.}, 154(3-4):369--408, 2023.

\bibitem{Braess2009}
Dietrich Braess, Veronika Pillwein, and Joachim Sch{\"o}berl.
\newblock Equilibrated residual error estimates are p-robust.
\newblock {\em Comput. Methods Appl. Mech. Engrg.}, 198(13-14):1189--1197, 2009.

\bibitem{Cangiani2017}
Andrea Cangiani, Emmanuil~H. Georgoulis, Tristan Pryer, and Oliver~J. Sutton.
\newblock A posteriori error estimates for the virtual element method.
\newblock {\em Numer. Math.}, 137(4):857--893, 2017.

\bibitem{ChaumontFrelet2025}
Th{\'e}ophile Chaumont-Frelet.
\newblock A posteriori error estimates for the finite element discretization of second-order {PDE}s set in unbounded domains.
\newblock arXiv preprint arXiv:2503.22297, 2025.

\bibitem{ChaumontGedickeMascotto2025}
Th{\'e}ophile Chaumont-Frelet, Joscha Gedicke, and Lorenzo Mascotto.
\newblock Generalised gradients for virtual elements and applications to a posteriori error analysis.
\newblock {\em arXiv preprint}, arXiv:2408.03148, 2025.

\bibitem{Chen2014}
Long Chen, Junping Wang, and Xiu Ye.
\newblock A posteriori error estimates for weak {G}alerkin finite element methods for second order elliptic problems.
\newblock {\em J. Sci. Comput.}, 59(2):496--511, 2014.

\bibitem{Chen2022}
Xinjiang Chen and Yanqiu Wang.
\newblock A conforming quadratic polygonal element and its application to {S}tokes equations.
\newblock {\em J. Comput. Math.}, 40(4):624--648, 2022.

\bibitem{CockburnGopalakrishnanLazarov2009}
Bernardo Cockburn, Jayadeep Gopalakrishnan, and Raytcho Lazarov.
\newblock Unified hybridization of discontinuous {G}alerkin, mixed, and continuous {G}alerkin methods for second order elliptic problems.
\newblock {\em SIAM J. Numer. Anal.}, 47(2):1319--1365, 2009.

\bibitem{DassiGedickeMascotto2022}
Franco Dassi, Joscha Gedicke, and Lorenzo Mascotto.
\newblock Adaptive virtual element methods with equilibrated fluxes.
\newblock {\em Appl. Numer. Math.}, 173:249--278, 2022.

\bibitem{DiPietro2023}
Daniele~A Di~Pietro and J{\'e}r{\^o}me Droniou.
\newblock An arbitrary-order discrete de {R}ham complex on polyhedral meshes: Exactness, {P}oincar{\'e} inequalities, and consistency.
\newblock {\em Foundations of Computational Mathematics}, 23(1):85--164, 2023.

\bibitem{DiPietroErn2015}
Daniele~A Di~Pietro and Alexandre Ern.
\newblock Hybrid high-order methods for variable-diffusion problems on general meshes.
\newblock {\em Comptes Rendus Math{\'e}matique}, 353(1):31--34, 2015.

\bibitem{Du2022}
Xiaoxiao Du, Wei Wang, Gang Zhao, Jiaming Yang, Mayi Guo, and Ran Zhang.
\newblock Virtual element method with adaptive refinement for problems of two-dimensional complex topology models from an engineering perspective.
\newblock {\em Comput. Mech.}, 70(3):581--606, 2022.

\bibitem{ErnVohralik2015}
Alexandre Ern and Martin Vohral\'ik.
\newblock Polynomial-degree-robust a posteriori estimates in a unified setting for conforming, nonconforming, discontinuous {G}alerkin, and mixed discretizations.
\newblock {\em SIAM J. Numer. Anal.}, 53(2):1058--1081, 2015.

\bibitem{Floater2015ref}
Michael~S. Floater.
\newblock Generalized barycentric coordinates and applications.
\newblock {\em Acta Numerica}, 24:161--214, 2015.

\bibitem{Floater2010}
Michael~S. Floater and Ji{\v{r}}{\'\i} Kosinka.
\newblock On the injectivity of {W}achspress and mean value mappings between convex polygons.
\newblock {\em Adv. Comput. Math.}, 32(2):163--174, 2010.

\bibitem{Floater2016}
Michael~S. Floater and Ming-Jun Lai.
\newblock Polygonal spline spaces and the numerical solution of the {P}oisson equation.
\newblock {\em SIAM J. Numer. Anal.}, 54(2):797--824, 2016.

\bibitem{Gillette2012}
Andrew Gillette, Alexander Rand, and Chandrajit Bajaj.
\newblock Error estimates for generalized barycentric interpolation.
\newblock {\em Adv. Comput. Math.}, 37(3):417--439, 2012.

\bibitem{Hormann2008}
Kai Hormann and Natarajan Sukumar.
\newblock Maximum entropy coordinates for arbitrary polytopes.
\newblock {\em Comput. Graph. Forum}, 27(5):1513--1520, 2008.

\bibitem{Hoshina2018}
Thom{\'a}s Y.~S. Hoshina, Ivan F.~M. Menezes, and Anderson Pereira.
\newblock A simple adaptive mesh refinement scheme for topology optimization using polygonal meshes.
\newblock {\em J. Braz. Soc. Mech. Sci. Eng.}, 40(7):348, 2018.

\bibitem{Joshi2007}
Pushkar Joshi, Mark Meyer, Tony DeRose, Brian Green, and Tom Sanocki.
\newblock Harmonic coordinates for character articulation.
\newblock In {\em {ACM} {SIGGRAPH} 2007 Papers}, pages 71:1--71:9, New York, 2007. ACM.

\bibitem{Lai2016}
Ming-Jun Lai and George Slavov.
\newblock On recursive refinement of convex polygons.
\newblock {\em Comput. Aided Geom. Design}, 45:83--90, 2016.

\bibitem{Li2019}
Hengguang Li, Lin Mu, and Xiu Ye.
\newblock A posteriori error estimates for the weak {G}alerkin finite element methods on polytopal meshes.
\newblock {\em Commun. Comput. Phys.}, 26(2):558--578, 2019.

\bibitem{Li2018SINUM}
Yu-Wen Li.
\newblock Global superconvergence of the lowest-order mixed finite element on mildly structured meshes.
\newblock {\em SIAM J. Numer. Anal.}, 56(2):792--815, 2018.

\bibitem{Li2021JSCb}
Yuwen Li.
\newblock Recovery-based a posteriori error analysis for plate bending problems.
\newblock {\em J. Sci. Comput.}, 88(3):Paper No. 77, 26, 2021.

\bibitem{Li2025arxiv}
Yuwen Li.
\newblock Some p-robust a posteriori error estimates based on auxiliary spaces.
\newblock {\em arXiv preprint}, arXiv:2511.06603, 2025.

\bibitem{LiShui2026}
Yuwen Li and Han Shui.
\newblock Smoother-type a posteriori error estimates for finite element methods.
\newblock {\em Comput. Methods Appl. Mech. Engrg.}, 453:Paper No. 118847, 21, 2026.

\bibitem{LiZikatanov2021CAMWA}
Yuwen Li and Ludmil Zikatanov.
\newblock A posteriori error estimates of finite element methods by preconditioning.
\newblock {\em Comput. Math. Appl.}, 91:192--201, 2021.

\bibitem{LiZikatanov2025mcom}
Yuwen Li and Ludmil Zikatanov.
\newblock Nodal auxiliary a posteriori error estimates.
\newblock {\em Math. Comp.}, 2025, DOI: 10.1090/mcom/4141.

\bibitem{Loop1989}
Charles~T. Loop and Tony~D. DeRose.
\newblock A multisided generalization of {B}\'ezier surfaces.
\newblock {\em ACM Trans. Graph.}, 8(3):204--234, 1989.

\bibitem{Malsch2004}
Elisabeth~Anna Malsch and Gautam Dasgupta.
\newblock Interpolations for temperature distributions: A method for all non-concave polygons.
\newblock {\em Int. J. Solids Struct.}, 41(8):2165--2188, 2004.

\bibitem{Meyer2002}
Mark Meyer, Alan Barr, Haeyoung Lee, and Mathieu Desbrun.
\newblock Generalized barycentric coordinates on irregular polygons.
\newblock {\em J. Graph. Tools}, 7(1):13--22, 2002.

\bibitem{NguyenXuan2017}
H.~Nguyen-Xuan.
\newblock A polytree-based adaptive polygonal finite element method for topology optimization.
\newblock {\em Internat. J. Numer. Methods Engrg.}, 110(10):972--1000, 2017.

\bibitem{PorwalSingla2025}
Kamana Porwal and Ritesh Singla.
\newblock A posteriori error analysis of hybrid high-order methods for the elliptic obstacle problem.
\newblock {\em J. Sci. Comput.}, 102(1):15, 2025.

\bibitem{Leon2014b}
Daniel~W. Spring, Sofie~E. Leon, and Glaucio~H. Paulino.
\newblock Unstructured polygonal meshes with adaptive refinement for the numerical simulation of dynamic cohesive fracture.
\newblock {\em Int. J. Fract.}, 189:33--57, 2014.

\bibitem{Tian2025}
Pengjie Tian and Yanqiu Wang.
\newblock Upper bounds of higher-order derivatives for {W}achspress coordinates on polytopes.
\newblock {\em IMA J. Numer. Anal.}, page draf063, 2025.

\bibitem{Verfurth2003}
R\"udiger Verf\"{u}rth.
\newblock A posteriori error estimates for finite element discretizations of the heat equation.
\newblock {\em Calcolo}, 40(3):195--212, 2003.

\bibitem{Wachspress1975}
Eugene~L. Wachspress.
\newblock {\em A rational finite element basis}, volume Vol. 114 of {\em Mathematics in Science and Engineering}.
\newblock Academic Press, Inc. [Harcourt Brace Jovanovich, Publishers], New York-London, 1975.

\bibitem{WangYe2014}
Junping Wang and Xiu Ye.
\newblock A weak {G}alerkin mixed finite element method for second order elliptic problems.
\newblock {\em Math. Comp.}, 83(289):2101--2126, 2014.

\bibitem{Wang2021}
Qiming Wang and Zhaojie Zhou.
\newblock Adaptive virtual element method for optimal control problem governed by general elliptic equation.
\newblock {\em J. Sci. Comput.}, 88(1):14, 2021.

\bibitem{Warren1996}
Joe Warren.
\newblock Barycentric coordinates for convex polytopes.
\newblock {\em Adv. Comput. Math.}, 6(1):97--108, 1996.

\bibitem{Warren2007}
Joe~D. Warren, Scott Schaefer, Anil~N. Hirani, and Mathieu Desbrun.
\newblock Barycentric coordinates for convex sets.
\newblock {\em Adv. Comput. Math.}, 27(3):319--338, 2007.

\bibitem{Weber2012}
Ofir Weber, Roi Poranne, and Craig Gotsman.
\newblock Biharmonic coordinates.
\newblock {\em Comput. Graph. Forum}, 31(8):2409--2422, 2012.

\bibitem{ZhangNaga2005}
Zhimin Zhang and Ahmed Naga.
\newblock A new finite element gradient recovery method: superconvergence property.
\newblock {\em SIAM J. Sci. Comput.}, 26(4):1192--1213, 2005.

\bibitem{Zienkiewicz1992b}
Olgierd~Cecil Zienkiewicz and Jianzhong Zhu.
\newblock The superconvergent patch recovery and a posteriori error estimates. part ii. error estimates and adaptivity.
\newblock {\em Internat. J. Numer. Methods Engrg.}, 33(7):1365--1382, 1992.

\end{thebibliography}

\end{document}